\documentclass{article}

\usepackage[utf8]{inputenc}
\usepackage[english]{babel}

\usepackage{geometry}
\usepackage{fancyhdr}
\usepackage{titlesec}
\usepackage{indentfirst}
\usepackage{float}

\usepackage[linktoc=all]{hyperref}
\usepackage{tikz}

\usepackage{amsmath,comment}
\usepackage{amssymb}
\usepackage{mathrsfs}
\usepackage{amsthm}
\usepackage{amsfonts}
\usepackage{ifthen}

\usepackage{cleveref}

\theoremstyle{plain}
\newtheorem{thm}{Theorem}[section]
\newtheorem{prop}[thm]{Proposition}
\newtheorem{lem}[thm]{Lemma}
\newtheorem{cor}[thm]{Corollary}
\newtheorem{defn}[thm]{Definition}
\newtheorem{prob}[thm]{Problem}

\newtheorem{introthm}{Theorem}

\newtheorem*{thm*}{Theorem}
\newtheorem*{cor*}{Corollary}

\theoremstyle{remark}

\newtheorem*{remark*}{Remark}
\newtheorem{remark}{Remark}

\DeclareMathOperator{\id}{id}

\DeclareMathOperator{\rank}{rk}

\newcommand{\bw}{\mathbf{w}}

\newcommand{\bbN}{\mathbb{N}}
\newcommand{\bbZ}{\mathbb{Z}}
\newcommand{\bbQ}{\mathbb{Q}}

\newcommand{\cA}{\mathcal{A}}

\newcommand{\cD}{\mathcal{D}}

\newcommand{\cG}{\mathcal{G}}

\newcommand{\cP}{\mathcal{P}}
\newcommand{\cR}{\mathcal{R}}

\newcommand{\fD}{\mathfrak{D}}
\newcommand{\fF}{\mathfrak{F}}

\newcommand{\rar}{\rightarrow}

\newcommand{\ol}[1]{\overline{#1}}
\newcommand{\ot}[1]{\widetilde{#1}}

\newcommand{\sgr}{\le}
\newcommand{\nor}{\trianglelefteq}
\newcommand{\gen}[1]{\langle #1 \rangle}
\newcommand{\pres}[2]{\langle #1 | #2 \rangle}
\newcommand{\im}[1]{\text{im}(#1)}

\newcommand{\divides}{\bigm|}

\newcommand{\FG}[1]{\mathrm{FG}(#1)}

\newcommand{\cyclicR}[1]{\cR_{=\bbZ}(#1)}
\newcommand{\notcyclicR}[1]{\cR_{\not=\bbZ}(#1)}
\newcommand{\aut}[1]{\mathrm{Aut}(#1)}
\newcommand{\out}[1]{\mathrm{Out}(#1)}
\newcommand{\outz}[1]{\mathrm{Out}_\mathrm{0}(#1)}

\newcommand{\outrel}[2]{\mathrm{Out}(#1 \text{ }\mathrm{rel}\text{ } #2)}
\newcommand{\outr}[1]{\mathrm{Out}_{\mathrm{r}}(#1)}
\newcommand{\perm}[1]{\mathrm{Perm}(#1)}
\newcommand{\Perm}{\mathrm{Perm}}

\newsavebox{\commentbox}
{\ifthenelse{\equal{\showcommentsbox}{yes}}%
{\footnotemark
        \begin{lrbox}{\commentbox}
        \begin{minipage}[t]{1.25in}\raggedright\sffamily\tiny
        \footnotemark[\arabic{footnote}]}
{\begin{lrbox}{\commentbox}}}%
{\ifthenelse{\equal{\showcommentsbox}{yes}}%
{\end{minipage}\end{lrbox}\marginpar{\usebox{\commentbox}}}
{\end{lrbox}}}

\title{Automorphisms of graphs of groups with cyclic edge groups}

\author{
Dario Ascari\\
{\small \textit{Department of Mathematics, University of the Basque Country,}}\\
{\small \textit{Barrio Sarriena, Leioa, 48940, Spain}}\\
{\small e-mail: \texttt{ascari.maths@gmail.com}}\\
\and
Montserrat Casals-Ruiz\\
{\small \textit{Ikerbasque - Basque Foundation for Science and Matematika Saila,}}\\
{\small \textit{UPV/EHU,  Sarriena s/n, 48940, Leioa - Bizkaia, Spain}}\\
{\small e-mail: \texttt{montsecasals@gmail.com}}\\
\and
Ilya Kazachkov\\
{\small \textit{Ikerbasque - Basque Foundation for Science and Matematika Saila,}}\\
{\small \textit{UPV/EHU,  Sarriena s/n, 48940, Leioa - Bizkaia, Spain}}\\
{\small e-mail: \texttt{ilya.kazachkov@gmail.com}}
}

\date{}

\begin{document}

\maketitle

\begin{abstract}
We describe the outer automorphism group of a one-ended fundamental group of a graph of groups, when edge groups are cyclic, and vertex groups are torsion-free with cyclic centralizers. We show that in this case the outer automorphism group is virtually built from the outer automorphisms of the vertex groups (fixing some elements), the outer automorphisms of some associated generalized Baumslag-Solitar groups, and generalized twists - partial conjugations by elements in the centralizers of some elliptic elements in the group. 
\end{abstract}


\section{Introduction}
The study of (outer) automorphisms of groups is a classical and central topic in group theory. Prominent examples include the automorphism groups of free abelian groups, free groups, and surface groups, namely $GL_n(\mathbb Z)$, $\out{F_n}$, and the mapping class group, respectively.

Beyond these classical cases, hyperbolic groups constitute another fundamental family. A key tool in understanding their outer automorphism groups is the JSJ decomposition, introduced by Rips and Sela as a group-theoretic analogue of the JSJ decomposition for 3-manifolds. For irreducible, closed, oriented 3-manifolds, the JSJ decomposition is uniquely determined in the strongest sense: it determines a unique (up to isotopy) minimal collection of disjoint, embedded, incompressible tori such that cutting the manifold along these tori results in components that are either atoroidal or Seifert-fibered, see \cite{Hat}. For one-ended torsion-free hyperbolic groups, a similar form of uniqueness can be achieved via the tree of cylinders, see \cite{GL11}. In this setting, the vertex groups in the resulting JSJ decomposition are of three types: cyclic groups, surface groups, or rigid hyperbolic groups. Sela showed that the outer automorphism group $\out{G}$ can be fully described in terms of this decomposition: it is virtually generated by Dehn twists along the edges and by mapping class groups associated to the surface-type vertex groups. More precisely, in \cite{Sel97} (see also \cite{Lev05}), the authors proved the following:

\begin{thm}[see Theorem 1.9 in \cite{Sel97} and Theorem 5.1 in \cite{Lev05}] \label{thm:sela} Let $G$ be a torsion-free one-ended hyperbolic group. There is an exact sequence
$$
1 \to \mathbb Z^n \to \outz{G} \to \prod MCG(\Sigma_v) \to 1,
$$
where $\outz{G}$ has finite index in $\out{G}$.
\end{thm} 

These ideas were further generalized for one-ended relatively hyperbolic groups by Guirardel and Levitt, see \cite{GL15}. In particular, for toral relatively hyperbolic groups $G$, they showed that $\out{G}$ is virtually built out of mapping class groups and subgroups of $GL_n(\mathbb Z)$ fixing certain basis elements. When more general parabolic groups are allowed, these subgroups of $GL_n(\mathbb Z)$ have to be replaced by McCool groups: automorphisms of parabolic groups acting trivially (i.e. by conjugation) on certain subgroups. They also determine when $\out{G}$ is infinite and showed that the infiniteness of $\out{G}$ comes from the existence of a splitting with infinitely many twists, or having a vertex group that is maximal parabolic with infinitely many automorphisms acting trivially on incident edge groups.

However, in general, the JSJ decomposition is far from being unique and  can exhibit considerable flexibility, if not in the pieces obtained through cutting, at least in the way they are glued together. Given a splitting $T$ of a group $G$ (i.e. an action of $G$ on a tree $T$), the \textit{deformation space} $\fD_G(T)$ is the family of all splittings of $G$ which can be obtained from $T$ using \textit{elementary deformations}, see \cite{For02}. The easiest examples of one-ended groups that exhibit infinitely many different gluing patterns, i.e. have an infinite deformation space, belong to the class of fundamental groups of graphs of groups with infinite cyclic vertex and edge groups, known as \emph{generalized Baumslag-Solitar groups, GBS groups or GBSs}. Due to the complexity of studying infinite deformation spaces, the study of automorphisms of fundamental groups of graphs of groups has so far been mainly restricted to the study of the subgroup of automorphisms that preserve a given decomposition, see \cite{BJ96, Lev05}. 

\medskip

In this paper, we study the outer automorphism group of (some families of) fundamental groups of graphs of groups with cyclic edge groups and show that it is virtually built from the outer automorphisms of the vertex groups, the outer automorphisms groups of some canonically determined GBS groups, and generalized twists (partial conjugations by elements in the centralizers of some elliptic elements).

Let us now be more precise. Let $G$ be a finitely presented one-ended torsion-free group; we denote by $\out{G}$ the outer automorphism group of $G$. We consider a cyclic JSJ decomposition for $G$, and we denote by $\outz{G}\sgr\out{G}$ the finite index (normal) subgroup of automorphisms that do not permute the vertices of the JSJ. We denote by $\outr{G}\sgr\outz{G}$ the (normal) subgroup of automorphisms that act as the identity on the vertex groups of the JSJ; this means that we are factoring away the internal automorphisms of the vertices. Although we do not make any hyperbolicity assumptions on the group $G$, the relationship between the groups $\out{G},\outz{G},\outr{G}$ is a natural extension of the hyperbolic case (see \cite{Sel97,Lev05}), and can be summarized in the following result.

\begin{introthm}\label{introthm:preliminary-sequences}
The following sequences are exact
\[
    1\rar\outz{G}\rar\out{G}\rar\text{{\rm(}finite group of permutations{\rm)}}
\]
and
\[
    1\rar\outr{G}\rar\outz{G} \rar \left(\prod\out{G_v}\right)\times\text{{\rm(}finitely generated abelian group{\rm)}},
\]
where the product runs over the vertex groups of the JSJ that are not isomorphic to $\bbZ$.
\end{introthm}

We note that in \Cref{introthm:preliminary-sequences} there are finitely generated abelian groups that do not occur under the hyperbolicity assumption. These abelian groups appear only if $G$ contains solvable Baumslag-Solitar groups. More precisely, a cyclic vertex group might be conjugated to a proper subgroup of itself, and in this case, the outer automorphisms that send the generator to a proper power are encoded (up to conjugation) in a finitely generated abelian group. Recall that in $\outr{G}$, we only consider outer automorphisms that fix the generator of the vertex group (up to conjugacy).

We further study the structure of the group $\outr{G}$ and show that this group is essentially determined by those of the GBSs, as formulated in the following

\begin{introthm}[Theorem \ref{thm:description-outrel}]\label{introthm:description-outrel}
    Let $\fF$ be a family of finitely presented torsion-free groups with cyclic centralizers. Let $\cG$ be a graph of groups with vertex groups in $\fF$ and with infinite cyclic edge groups. Suppose that $G=\pi_1(\cG)$ is freely indecomposable. Then there is an exact sequence
    \[
        1\rar Z\rar\prod_{\rho\in\cR}C(\rho)\xrightarrow{\tau} \outr{G}\xrightarrow{\sigma}\outr{G_1}\times\outr{G_2}\times\dots \times\outr{G_s}\times(D_\infty)^t \to 1
    \]
    for some finite set of elliptic elements $\cR\subseteq G$ and their centralizers $C(\rho)$ in $G$; for some finitely generated free abelian group $Z$; for some integers $s,t\ge0$, and for some GBSs $G_1,\dots,G_s$ {\rm(}here $D_\infty$ is the infinite dihedral group{\rm)}. 
\end{introthm}

The GBSs are canonically associated with the decomposition of $G$ and can be roughly described as follows. We take a JSJ decomposition of $G$, for each non-cyclic vertex stabilizer, we consider maximal infinite cyclic subgroups corresponding to conjugacy classes in its peripheral structure. We take a non-cyclic vertex stabilizer $G_v$, and we split that vertex into several maximal infinite cyclic subgroups, one for each conjugacy class in its peripheral structure. An edge that was previously glued on a certain conjugacy class, will now be glued on the corresponding infinite cyclic group. We reiterate the procedure for all non-cyclic vertices of the JSJ decomposition, and we obtain a (possibly disconnected) GBS $\Delta$. We define $G_1,\dots,G_k$ as the fundamental groups of the connected components of $\Delta$.

We point out that the above Theorem \ref{introthm:description-outrel} works in particular for one-ended fundamental groups of graphs of torsion-free hyperbolic groups with cyclic edge groups, even when they are not necessarily hyperbolic; in this case, and up to finite index, the only difference between $\out{G}$ and $\outr{G}$ is given by the mapping class groups of the flexible surfaces in the JSJ decomposition as stated in the following.

\begin{cor*}[\Cref{cor:graph_hyperbolic_group}]
Let $\cG$ be a graph of groups with torsion-free hyperbolic vertex groups and with cyclic edge groups. Suppose that $G=\pi_1(\cG)$ is freely indecomposable. Then there are exact sequences
    \[
    \begin{array}{c}
    1 \to \outr{G} \to \outz{G} \to \prod MCG(\Sigma_v)\times (\hbox{finite group})\times (\hbox{fg abelian});\\[0.2cm]
        1\rar Z \rar\prod_{\rho\in\cR}C(\rho) \xrightarrow{\tau} \outr{G}\xrightarrow{\sigma}\outr{G_1}\times\outr{G_2}\times\dots \times\outr{G_s}\times(D_\infty)^t \to 1
    \end{array}
    \]
    for some finite set of elliptic elements $\cR\subseteq G$, for some finitely generated free abelian group $Z$, for some integers $s,t\ge0$, and for some GBSs $G_1,\dots,G_s$ {\rm(}here $D_\infty$ is the infinite dihedral group{\rm)}. 
    
\end{cor*}

An important corollary from the combination theorem of Bestvina and Feighn, see \cite{BF92}, characterizes when a graph of groups $G$ with torsion-free hyperbolic vertex groups and cyclic edge groups is again hyperbolic, and it is precisely when the group $G$ does not contain (non-cyclic) Baumslag-Solitar subgroups. In this setting, the associated GBS groups $G_i$ are cyclic and so the relative outer automorphism group $\outr{G_i}$ is trivial, and the centralizers of elements $C(\rho)$ are also cyclic. Therefore, $\outr{G}$ is free abelian and we recover the result for torsion-free hyperbolic groups, namely Theorem \ref{thm:sela}. 

\begin{remark}
When studying outer automorphisms of finitely generated torsion-free groups with infinitely many ends in relation to their splittings over the trivial group, the main tool is a generalization of Whitehead moves, inspired by techniques originally developed for the study of outer automorphisms of free groups \cite{CZ84a,CZ84b,HL92}. However, there is no direct algebraic relationship between the outer automorphism group of such a group $G$ and that of a free group. In contrast, in the context of cyclic JSJ splittings for finitely presented one-ended groups, we show that the relationship between the outer automorphism group of $G$ and those of the associated GBSs is algebraic: the outer automorphism group $\outr{G}$ surjects onto the direct product of the outer automorphism groups $\outr{G_i}$ of the associated GBSs. These results highlight once again the central importance of understanding GBSs and the structure of their outer automorphism groups, since they serve as foundational building blocks for the study of automorphism groups of broader classes of groups.
\end{remark}

Our results raise a natural question. As in the original work of Sela for hyperbolic groups, one of the standing assumptions is that the group be torsion-free. Guirardel and Levitt gave a description of the outer automorphism groups of relatively hyperbolic groups without any assumption on the torsion. Therefore, we arrive at the following

\begin{prob}
Describe the outer automorphism group of a one-ended fundamental group of a graph of groups, with virtually cyclic edge groups. Is it possible to obtain an analogue of {\rm \Cref{introthm:description-outrel}}? If so, which assumptions on the centralizers in the vertex groups are required?
\end{prob}

In particular, it would be interesting to know whether it is possible to extend our results to one-ended graphs of groups with hyperbolic vertex groups and virtually cyclic edge groups.

\medskip

\textbf{Organization of the paper.} In Section \ref{sec:preliminaries}, we introduce some background. In Section \ref{sec:GBS-centralizers}, we give a description of centralizers of elements in GBS groups. In Section \ref{sec:relautos}, we introduce the group of relative outer automorphisms of a fundamental group of the graph of groups $G$. In Section \ref{sec:comparing-automorphism-groups}, we describe the relationship between the outer automorphism of $G$ and its relative outer automorphism, and we give a description of the relative outer automorphism groups of $G$ in terms of the outer automorphisms of GBSs.

\subsection*{Acknowledgements}

This work was supported by the Basque Government grant IT1483-22. The second author was supported by the Spanish Government grant PID2020-117281GB-I00, partly by the European Regional Development Fund (ERDF), the MICIU /AEI /10.13039/501100011033 / UE grant PCI2024-155053-2.

\section{Preliminaries}\label{sec:preliminaries}

In this section, we recall the fundamental concepts and set up the notation that we are going to use on graphs of groups and JSJ decompositions. We point out that Section \ref{sec:preliminaries} is completely expository, based on the already existing literature, and the (few) proofs provided are well-known standard arguments. For the Bass-Serre correspondence, we refer the reader to \cite{Ser77}. For a general and in-depth treatment of the theory of JSJ decomposition, we refer the reader to \cite{GL17}. We also remark that we are mostly going to apply the results of this section to finitely presented torsion-free groups and to splittings over the trivial or infinite cyclic group.

\subsection{Graphs}

We consider graphs as combinatorial objects, following the notation of \cite{Ser77}. A \textbf{graph} is a quadruple $\Gamma=(V,E,\ol{\cdot},\iota)$ consisting of a set $V=V(\Gamma)$ of \textit{vertices}, a set $E=E(\Gamma)$ of \textit{edges}, a map $\ol{\cdot}:E\rar E$ called \textit{reverse} and a map $\iota:E\rar V$ called \textit{initial vertex}; we require that for every edge $e\in E$, we have $\ol{e}\not=e$ and $\ol{\ol{e}}=e$. For an edge $e\in E$, we denote by $\tau(e)=\iota(\ol{e})$ the \textit{terminal vertex} of $e$. A \textbf{path} in a graph $\Gamma$, with \textit{initial vertex} $v\in V(\Gamma)$ and \textit{terminal vertex} $v'\in V(\Gamma)$, is a sequence $\sigma=(e_1,\dots,e_\ell)$ of edges $e_1,\dots,e_\ell\in E(\Gamma)$ for some integer $\ell\ge0$, with the conditions $\iota(e_1)=v$ and $\tau(e_\ell)=v'$ and $\tau(e_i)=\iota(e_{i+1})$ for $i=1,\dots,\ell-1$. A path $(e_1,\dots,e_\ell)$ is \textbf{reduced} if $e_{i+1}\not=\ol{e}_i$ for $i=1,\dots,\ell-1$. A graph is \textbf{connected} if for every pair of vertices, there is a path that goes from one to the other. For a connected graph $\Gamma$, we define its \textbf{rank} $\rank{\Gamma}\in\bbN\cup\{+\infty\}$ as the rank of its fundamental group (which is a free group). A graph is a \textbf{tree} if for every pair of vertices, there is a unique reduced path going from one to the other.

\subsection{Graphs of groups}\label{sec:graphs-of-groups}

\begin{defn}[Graph of groups]\label{def:graph-of-groups}
A \textbf{graph of groups} is a quadruple
$$\cG=(\Gamma,\{G_v\}_{v\in V(\Gamma)},\{G_e\}_{e\in E(\Gamma)},\{\psi_e\}_{e\in E(\Gamma)})$$
consisting of a connected graph $\Gamma$, a group $G_v$ for each vertex $v\in V(\Gamma)$, a group $G_e$ for every edge $e\in E(\Gamma)$ with the condition $G_e=G_{\ol{e}}$, and an injective homomorphism $\psi_e:G_e\rar G_{\tau(e)}$ for every edge $e\in E(\Gamma)$.
\end{defn}

\begin{defn}[Universal group]\label{def:universal-group}
Let $\cG=(\Gamma,\{G_v\},\{G_e\},\{\psi_e\})$ be a graph of groups. Define the \textbf{universal group} $\FG{\cG}$ as the quotient of the free product $(*_{v\in V(\Gamma)}G_v)*F(E(\Gamma))$ by the relations
\[
    \ol{e}=e^{-1} \qquad\qquad \psi_{\ol{e}}(g)\cdot e=e\cdot\psi_e(g)
\]
for $e\in E(\Gamma)$ and $g\in G_e$.
\end{defn}

A \textbf{sound writing} for an element $x\in\FG{\cG}$ is given by a path $(e_1,\dots,e_\ell)$ in $\Gamma$ and by elements $g_i\in G_{\tau(e_i)}=G_{\iota(e_{i+1})}$ for $i=0,\dots,\ell$, such that $x=g_0e_1g_1e_2g_2\dots g_{\ell-1}e_\ell g_\ell$ in $\FG{\cG}$. An element $x\in\FG{\cG}$ is called \textbf{sound} if it admits a sound writing; in that case, the vertices $v=\iota(e_1),v'=\tau(e_\ell)\in V(\Gamma)$ do not depend on the writing (nor does the homotopy class relative to the endpoints of the path $(e_1,\dots,e_\ell)$ in $\Gamma$), and we say that $x$ is \textbf{$(v,v')$-sound}. We denote by $\pi_1(\cG,v,v')$ the set of all $(v,v')$-sound elements of $\FG{\cG}$.

Given a vertex $v_0\in V(\Gamma)$, we have a subgroup $\pi_1(\cG,v_0):=\pi_1(\cG,v_0,v_0)\sgr\FG{\cG}$ which we call \textbf{fundamental group} of $\cG$ with basepoint $v_0$; different choices of the basepoint $v_0$ produce conjugate (and in particular isomorphic) subgroups of $\FG{\cG}$. Given an element $g$ in a vertex group $G_v$, we can take a path $(e_1,\dots,e_\ell)$ from $v_0$ to $v$: we define the conjugacy class $[g]:=[e_1\dots e_\ell g\ol{e}_\ell\dots \ol{e}_1]$ in $\pi_1(\cG,v_0)$, and notice that this does not depend on the chosen path.

Let $x\in\FG{\cG}$ be sound, and suppose that we are given a sound writing $x=g_0e_1g_1e_2\dots g_{\ell-1}e_\ell g_\ell$. Suppose that for some $1\le i\le \ell-1$ we have that $e_{i+1}=\ol{e}_i$ 
and $g_i\in\im{\psi_{e_i}}$, then we can produce another sound writing for $x$, namely
\[
    x=g_0e_1g_1\dots  e_{i-1}[g_{i-1}\cdot \psi_{\ol{e}_i}(\psi_{e_i}^{-1}(g_i))\cdot g_{i+1}]e_{i+2}\dots g_{\ell-1}e_\ell g_\ell.
\]
We say that the second writing is obtained from the first by a \textbf{reduction}. A sound writing for $x$ is called \textbf{reduced} if no reduction can be performed to it. It is immediate to see that every sound element admits a reduced sound writing, which can be obtained by performing a sequence of reductions; the following lemma deals with uniqueness.

\begin{lem}[Britton's lemma \cite{Br63}]\label{lem:sound-writing}
    Let $\cG$ be a graph of groups and let $x\in\FG{\cG}$ be a sound element. Let $x=g_0e_1g_1\dots e_\ell g_\ell=g_0'e_1'g_1'\dots e_{\ell'}'g_{\ell'}'$ be two reduced sound writings. Then we have $\ell=\ell'$ and $e_i=e_i'$ for $i=1,\dots,\ell$. Moreover, there are elements $h_i\in G_{e_i}$ for $i=1,\dots,\ell$, such that $g_i'=\psi_{e_i}(h_i)\cdot g_i\cdot \psi_{\ol{e}_{i+1}}(\ol{h}_{i+1})$ for $i=0,\dots,\ell$.
\end{lem}
\begin{proof}
    See Section IV.2 of \cite{LS77}.
\end{proof}

\subsection{Bass-Serre correspondence}

An action of a group $G$ on a tree $T$ is called \textbf{without inversions} if $ge\not=\ol{e}$ of any $g\in G$ and $e\in E(T)$. If $G$ acts on a tree $T$, then we can take the first barycentric subdivision $T^{(1)}$ of $T$, and the action of $G$ on $T^{(1)}$ is without inversions; thus it is not restrictive to limit ourselves to actions without inversions. \textit{From now on, every action of a group $G$ on a tree $T$ is assumed to be without inversions.} Given an action of a group $G$ on a tree $T$, we denote by $G_{\ot v}\sgr G$ the stabilizer of a vertex $\ot{v}\in V(T)$ and with $G_{\ot e}\sgr G$ the stabilizer of an edge $\ot{e}\in E(T)$.

Let $G$ be a group acting on a tree $T$ and let $\Gamma$ be the quotient graph. We make the following choices: for every vertex $v\in V(\Gamma)$ we choose a lifting $\ot v\in V(T)$; for every edge $e\in E(\Gamma)$ we choose a lifting $\ot e\in E(T)$; for every edge $e\in E(\Gamma)$ with $\tau(e)=v$ we choose an element $a_e\in G$ such that $\tau(a_e\cdot\ot e)=\ot v$. We define the graph of groups $\cG$ with underlying graph $\Gamma$; at each vertex $v\in V(\Gamma)$ we put the stabilizer $G_{\ot v}$ of the chosen lifting; at every edge $e\in E(\Gamma)$ we put the stabilizer $G_{\ot e}$ of the chosen lifting; for every edge $e\in E(\Gamma)$ we define the map $\psi_e:G_{\ot e}\rar G_{\ot v}$ given by $x\mapsto a_exa_e^{-1}$. In addition, we define the homomorphism $\Phi:\FG{\cG}\rightarrow G$, consisting of the natural inclusions $G_{\ot v}\rightarrow G$ for $v\in V(\Gamma)$, and given by $\Phi(e)=a_{\ol e}^{-1}a_e$ for $e\in E(\Gamma)$.

If $\cG=(\Gamma,\{G_v\},\{G_e\},\{\psi_e\})$ and $\cG'=(\Gamma,\{G_v'\},\{G_e'\},\{\psi_e'\})$ are graphs of groups obtained from the same action on a tree, but performing different choices, then there are isomorphisms $G_v\rar G_v'$ and $G_e\rar G_e'$ such that all the diagrams with the maps $\psi_e,\psi_e',\Phi,\Phi'$ commute (up to conjugation). Thus, the graph of groups $\cG$ and the homomorphism $\Phi$ are essentially unique (up to changing the maps $\psi_e:G_e\rar G_{\tau(e)}$ by post-composing with conjugation in the vertex group $G_{\tau(e)}$; the conjugating element is allowed to depend on the chosen edge $e$).

\begin{prop}[Bass-Serre]\label{prop:Bass-Serre}
    For a group $G$, we have the following:
    \begin{itemize}
        \item For every action of $G$ on a tree $T$, the corresponding homomorphism $\Phi:\FG{\cG}\rightarrow G$ induces an isomorphism $\Phi:\pi_1(\cG)\rar G$ on the fundamental group (for all basepoints).
        \item Suppose that we are given a graph of groups $\cG'$ and a homomorphism $\Phi':\FG{\cG'}\rar G$ inducing an isomorphism $\Phi':\pi_1(\cG')\rar G$ on the fundamental group. Then there is an action of $G$ on a tree $T$, inducing a graph of groups $\cG$ and a homomorphism $\Phi:\FG{\cG}\rar G$, such that $\cG$ is isomorphic to $\cG'$ (the isomorphism commuting with the maps $\Phi,\Phi'$). The action is unique up to $G$-equivariant isomorphism.
    \end{itemize}
\end{prop}
\begin{proof}
    See \cite{Ser77}.
\end{proof}

An action of a group $G$ on a tree $T$ will be called a \textbf{splitting} of $G$. In virtue of the Bass-Serre correspondence of Proposition \ref{prop:Bass-Serre}, we will equivalently say that a \textbf{splitting} of $G$ is a graph of groups $\cG$ together with a homomorphism $\Phi:\FG{\cG}\rar G$ inducing an isomorphism $\Phi:\pi_1{\cG}\rar G$ on the fundamental group.

\subsection{Domination and deformation spaces}\label{sec:deformation-spaces}

Let $G$ be a group acting on a tree $T$. We say that a subgroup $H\sgr G$ is \textbf{elliptic in $T$} if the induced action of $H$ on $T$ has a global fixed point; this means that $H$ is contained in the stabilizer $G_v$ for some vertex $v\in V(T)$. Notice that a subgroup $H$ is elliptic in $T$ if and only if all its conjugates are elliptic in $T$.

Let $G$ be a group acting on two trees $T,T'$. We say that $T$ \textbf{dominates} $T'$ if every vertex stabilizer of $T$ is elliptic in $T'$; this is equivalent to requiring that there is a $G$-equivariant map $f:T\rar T'$ sending each vertex to a vertex and each edge to a (possibly trivial) edge path.

Let $G$ be a group acting on a tree $T$. The \textbf{deformation space} $\fD_G(T)$ is the set of actions of $G$ on a tree $T'$ such that $T$ dominates $T'$ and $T'$ dominates $T$ (considered up to $G$-equivariant isomorphism). The terminology \textit{deformation space} is due to the fact that two trees lie in the same deformation space if and only if it is possible to pass from one to the other by a sequence of certain moves, called \textit{deformations}. We refer the reader to \cite{For02} and \cite{BF91} for further details.

\subsection{Universally elliptic splittings and JSJ decomposition}

Let $\cA$ be a family of subgroups of $G$ which is closed under taking subgroups. An \textbf{$\cA$-splitting} is an action of $G$ on a tree $T$ such that all edge stabilizers belong to $\cA$. Let $T$ be an $\cA$-splitting: we say that an edge stabilizer $G_e$, for $e\in E(T)$, is \textbf{$\cA$-universally elliptic} if it is elliptic in every $\cA$-splitting; we say that $T$ is \textbf{$\cA$-universally elliptic} if all its edge stabilizers are $\cA$-universally elliptic. Roughly speaking, this means that the splitting $T$ is ``canonical", i.e. it is in some sense compatible with every other $\cA$-splitting $T'$.

\begin{defn}[JSJ decomposition]
    A \textbf{JSJ decomposition} for $G$ over $\cA$ is an $\cA$-universally elliptic tree $T$ which is maximal by domination {\rm(}i.e. if $T'$ is $\cA$-universally elliptic and $T'$ dominates $T$, then $T$ dominates $T'${\rm)}.
\end{defn}

\begin{thm}[Existence of JSJ]\label{thm:JSJ-existence}
    Let $G$ be a finitely presented group. Let $\cA$ be a family of subgroups of $G$ which is closed under taking subgroups. Then $G$ has a JSJ decomposition over $\cA$. 
\end{thm}
\begin{proof}
    This is Theorem 2.16 in \cite{GL17}. 
\end{proof}

We point out that the proof of Theorem \ref{thm:JSJ-existence} is based on Dunwoody's accessibility (see \cite{Dun85}). Theorem \ref{thm:JSJ-existence} becomes false if we only require $G$ to be finitely generated (see \cite{Dun93}). It is immediate from the definition that, if a JSJ decomposition exists, then the set of all JSJ decompositions is a deformation space.

\subsection{Deformation spaces and outer automorphisms}

Let $G$ be a group acting on a tree $T$ and let $\theta:G\rar G$ be an automorphism. Then we can define an action of $G$ on a tree $\theta(T)$ as follows: we set $\theta(T)=T$ as trees, and for $t\in T$ we define the action $g\cdot_{\theta(T)}t=\theta^{-1}(g)\cdot_Tt$. With this definition, we have $(\theta'\circ\theta)(T)=\theta'(\theta(T))$, i.e. this is an action of $\aut{G}$ on the family of all actions of $G$ on trees. If $\theta$ is an inner automorphism, then $\theta(T)$ is isomorphic to $T$; this means that if we consider trees up to $G$-equivariant isomorphism, we obtain a well-defined action of $\out{G}$.

If $\cA$ is a family of subgroups of $G$ which is invariant under automorphisms of $G$, and if $G$ has a JSJ decomposition $J$ over $\cA$, then we obtain an action of $\out{G}$ on the deformation space $\fD_G(J)$. This is true for example when taking families such as $\cA=\{1\}$ or $\cA=\{\hbox{finite subgroups}\}$ or $\cA=\{\hbox{virtually cyclic subgroups}\}$. We will be interested in torsion-free one-ended groups, with the family $\cA=\{\hbox{infinite cyclic subgroups}\}$.

\subsection{Generalized Baumslag-Solitar groups}

In this section, we recall the definition of a generalized Baumslag-Solitar group and the standard way of representing them.

\begin{defn}[GBS group]
A \textbf{GBS graph of groups} is a finite graph of groups such that each vertex group and each edge group is $\bbZ$. A \textbf{generalized Baumslag-Solitar group} is a group $G$ isomorphic to the fundamental group of some GBS graph of groups.
\end{defn}

\begin{thm}[Forester \cite{For03}]\label{thm:GBS-JSJ}
    Let $\cG$ be a GBS graph of groups and suppose that $\pi_1(\cG)\not\cong\bbZ,\bbZ^2,K$ {\rm(}the Klein bottle group{\rm)}. Then $\cG$ is a JSJ decomposition for $\pi_1(\cG)$ over the family of its cyclic subgroups.
\end{thm}
\begin{proof}
    See \cite{For03} (and also Section 3.5 in \cite{GL17}).
\end{proof}

We next recall the representation of generalized Baumslag-Solitar groups in terms of labeled graphs.

\begin{defn}[GBS graph]
A \textbf{GBS graph} is a pair $(\Gamma,\psi)$ where $\Gamma$ is a finite graph and $\psi:E(\Gamma)\rar\bbZ\setminus\{0\}$ is a function.
\end{defn}

Given a GBS graph of groups $\cG=(\Gamma,\{G_v\}_{v\in V(\Gamma)},\{G_e\}_{e\in E(\Gamma)},\{\psi_e\}_{e\in E(\Gamma)})$, the map $\psi_e:G_e\rar G_{\tau(e)}$ is an injective homomorphism $\psi_e:\bbZ\rar\bbZ$, and thus coincides with multiplication by a unique non-zero integer $\psi(e)\in\bbZ\setminus\{0\}$. We define the GBS graph associated to $\cG$ as $(\Gamma,\psi)$ associating to each edge $e$ the factor $\psi(e)$ characterizing the homomorphism $\psi_e$, see Figure \ref{fig:GBS-graph}. Giving a GBS graph of groups is equivalent to giving the corresponding GBS graph. In fact, the numbers on the edges are enough to reconstruct the injective homomorphisms and thus the graph of groups.

\begin{figure}[H]
\centering
\includegraphics[width=0.5\textwidth]{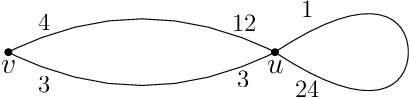}
\caption{In the figure we can see a GBS graph $(\Gamma,\psi)$ with two vertices $v,u$ and three edges $e_1,e_2,e_3$ (and their reverses). The edge $e_1$ goes from $v$ to $u$ and has $\psi(\ol{e}_1)=4$ and $\psi(e_1)=12$. The edge $e_2$ goes from $v$ to $u$ and has $\psi(\ol{e}_2)=\psi(\ol{e}_2)=3$. The edge $e_3$ goes from $u$ to $u$ and has $\psi(\ol{e}_3)=1$ and $\psi(e_3)=24$.}
\label{fig:GBS-graph}
\end{figure}

Let $\cG$ be a GBS graph of groups and let $(\Gamma,\psi)$ be the corresponding GBS graph. The universal group $\FG{\cG}$ has a presentation with generators $V(\Gamma)\cup E(\Gamma)$, the generator $v\in V(\Gamma)$ representing the element $1$ in $\bbZ=G_v$. The relations are given by $\ol{e}=e^{-1}$ and $u^{\psi(\ol{e})}e=ev^{\psi(e)}$ for every edge $e\in E(\Gamma)$ with $\iota(e)=u$ and $\tau(e)=v$.

\section{Centralizers in Generalized Baumslag-Solitar groups}\label{sec:GBS-centralizers}

In this section, we describe centralizers of elements in GBS groups. In Section \ref{sec:auto-centralizers}, we use centralizers of elliptic elements to define (relative) outer automorphisms of the fundamental group of a graph of groups with cyclic edge groups, which in a way are a generalization of Dehn twists along the edges of the splitting in the hyperbolic case.

In order to describe centralizers of elliptic elements, given a GBS graph, we define another (possibly infinite) (possibly disconnected) GBS graph, called the development (\Cref{def:GBS-development}). We show that the centralizer of an elliptic element is the fundamental group of a (possibly infinite) connected component of the development, see \Cref{prop:GBS-centralizer}.

\begin{defn}\label{def:centralizer-element}
    Let $\cG=(\Gamma,\{G_v\},\{G_e\},\{\psi_e\})$ be a graph of groups. For $v\in V(\Gamma)$ and $g\in G_v$, we denote by
    $$C(g)=\{c\in\pi_1(\cG,v) : cg=g c\}\sgr\pi_1(\cG,v)$$
    the centralizer of the element $g$.
\end{defn}

\begin{remark}
    Note that this is the centralizer in \textit{the whole fundamental group} of the graph of groups (and not just the centralizer inside $G_v$ itself). Also, centralizers of different elements are subgroups of the fundamental group at different basepoints.
\end{remark}

We are interested in giving a description of such centralizers in a GBS. In \cite{Dud18a}, the author describes centralizers of elements in a GBS when the corresponding GBS graph $\Gamma$ is a tree, and in \cite{Dud18b} the centralizers of elements in unimodular GBSs.

As we shall now see, the description of the centralizers of elliptic elements in a GBS is given in terms of infinite GBS, and for simplicity it is also convenient to allow them to be disconnected; for this reason, we will abuse the notation a little bit. Define an \textit{infinite disconnected GBS graph} as a pair $(\Gamma,\psi)$ where $\Gamma$ is a (possibly infinite, possibly disconnected) graph and $\psi:V(\Gamma)\rar\bbZ\setminus\{0\}$ is a function. Define an \textit{infinite disconnected GBS graph of groups} as a (possibly infinite) graph of groups with vertex groups and edge groups infinite cyclic, but allowing also the underlying graph to be disconnected. As for standard GBSs, there is a correspondence between (i) infinite disconnected GBS graphs, up to label-preserving isomorphism and (ii) infinite disconnected GBS graphs of groups with marked generators for the vertex groups and edge groups, up to isomorphism preserving the marked generators.

\begin{defn}\label{def:GBS-development}
    For a GBS graph $(\Gamma,\psi)$, define its \textbf{development} as the infinite disconnected GBS graph $(\ot\Gamma,\ot\psi)$ given by the following data {\rm(}see also {\rm Figure \ref{fig:development})}:
    \begin{enumerate}
        \item We define $V(\ot\Gamma)=V(\Gamma)\times(\bbZ\setminus\{0\})=\{(v,n) : v\in V(\Gamma)$ and $n\not=0$ integer$\}$.
        \item We define $E(\ot\Gamma)=E(\Gamma)\times(\bbZ\setminus\{0\})=\{(e,n) : e\in E(\Gamma)$ and $n\not=0$ integer$\}$.
        \item We define the reverse map $\ol{(e,n)}=(\ol{e},n)$.
        \item We define the endpoints $\tau_{\ot\Gamma}((e,n))=(\tau_\Gamma(e),n\cdot\psi(e))$ and $\iota_{\ot\Gamma}((e,n))=(\iota_\Gamma(e),n\cdot\psi(\ol{e}))$.
        \item We define the map $\ot\psi:E(\ot\Gamma)\rar\bbZ\setminus\{0\}$ given by $\ot\psi((e,n))=\psi(e)$.
    \end{enumerate}
\end{defn}

\begin{figure}[H]\label{fig:development}
    \centering
    \includegraphics[width=0.8 \textwidth]{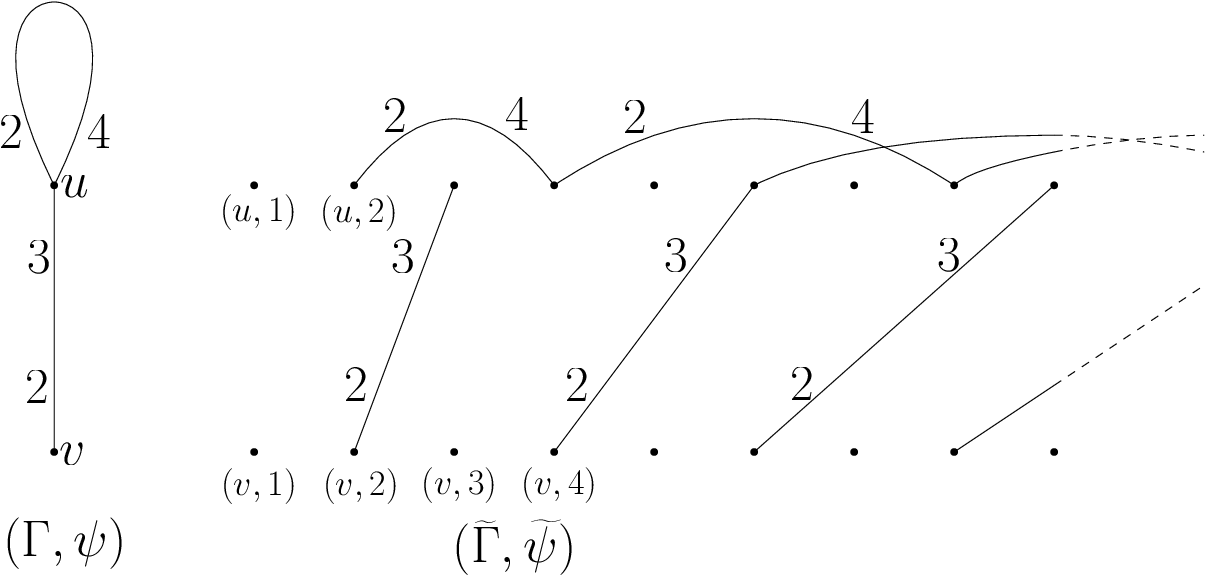}
    \caption{On the left a GBS graph $(\Gamma,\psi)$ with two vertices $v,u$. On the right its infinite disconnected development $(\ot\Gamma,\ot\psi)$ with vertices $(v,1),(v,2),\dots,(u,1),(u,2),\dots$ (vertices with negative coefficient have been omitted from the picture). For example, if $z\in G_u$ is the generator at $u\in V(\Gamma)$, then we can see that $C(z)\cong\bbZ$ (since the connected component of $(u,1)$ is trivial), $C(z^2)$ is an infinite GBS (since the connected component of $(u,2)$ is infinite), $C(z^3)\cong\pres{x,y}{x^3=y^2}$ (since the connected component if $(u,3)$ is a finite GBS), $C(z^4)\cong C(z^2)$ (since the connected component of $(u,4)$ coincides with the connected component of $(u,2)$, and in fact $z^4$ is conjugate to $z^2$).}
\end{figure}

Let $(\Gamma,\psi)$ be a GBS graph with corresponding GBS graph of groups $\cG=(\Gamma,\{G_v\},\{G_e\},\{\psi_e\})$. Let $(\ot\Gamma,\ot\psi)$ be its development, with corresponding infinite disconnected GBS graph of groups $\ot\cG=(\ot\Gamma,\{\ot G_{(v,n)}\},\{\ot G_{(e,n)}\},\{\ot\psi_{(e,n)}\})$. We can define the \textit{universal group} $\FG{\ot\cG}$ in the same way as in Definition \ref{def:universal-group}. For a vertex $(v,n)\in V(\ot\Gamma)$ we can define the \textit{fundamental group} $\pi_1(\ot\cG,(v,n))$ as in Section \ref{sec:graphs-of-groups}; the main difference being that vertices in distinct connected components can give non-isomorphic subgroups $\pi_1(\ot\cG,(v,n))\sgr\FG{\ot\cG}$. There is a natural map from $\FG{\ot\cG}\rar\FG{\cG}$ sending the generator of the vertex $G_{(v,n)}$ to the generator of $G_v$ for all $(v,n)\in V(\ot\Gamma)$, and sending the edge $(e,n)$ to $e$ for all $(e,n)\in E(\ot\Gamma)$; note that this is well-defined, as relations of $\FG{\ot\cG}$ are sent to relations of $\FG{\cG}$. This map sends reduced sound writings to reduced sound writings, and in particular, it induces an injective homomorphism $\pi_1(\ot\cG,(v,n))\rar\pi_1(\cG,v)$ for all vertices $(v,n)\in V(\ot\Gamma)$.

\begin{prop}[centralizers of elliptic elements]\label{prop:GBS-centralizer}
    Let $(\Gamma,\psi)$ be a GBS graph, with corresponding graph of groups $\cG$. Let $(\ot\Gamma,\ot\psi)$ be its development, with corresponding infinite disconnected graph of groups $\ot\cG$. Then, for every vertex $v\in V(\Gamma)$, and for every element $z^n\in G_v$ {\rm(}where $z$ is the generator of $G_v$ and $n\not=0$ is an integer{\rm)}, we have that $C(z^n)=\pi_1(\ot\cG,(v,n))\sgr\pi_1(\cG,v)$.
\end{prop}
\begin{proof}
    First take $\ot g\in \pi_1(\ot\cG,(v,n))$ with reduced sound writing
    $$\ot g=g_0(e_1,m_1)g_1(e_2,m_2)g_2\dots g_{\ell-1}(e_\ell,m_\ell)g_\ell$$
    and call $q_i=\frac{\psi(e_i)}{\psi(\ol{e}_i)} \in \bbQ\setminus\{0\}$. By definition of development, we must have that $nq_1\dots q_{i-1}=m_i\cdot\psi(\ol{e}_i)$ for $i=1,\dots,\ell$ (otherwise the writing is not sound) and $q_1\dots q_\ell=1$ (otherwise the path does not go back to the starting point). But now we look at the projection $g=g_0e_1g_1e_2g_2\dots g_{\ell-1}e_\ell g_\ell\in\pi_1(\cG,v)$ and we observe that
    \begin{gather}\notag
    \begin{split}
    z^ng=& z^ng_0e_1g_1e_2g_2\dots g_{\ell-1}e_\ell g_\ell= g_0e_1z^{nq_1}g_1e_2g_2\dots g_{\ell-1}e_\ell g_\ell= \\
    = &g_0e_1g_1e_2z^{nq_1q_2}g_2\dots g_{\ell-1}e_\ell g_\ell= \dots= g_0e_1g_1e_2g_2\dots g_{\ell-1}e_\ell z^{nq_1\dots q_\ell} g_\ell = gz^n
    \end{split}
    \end{gather}
    and thus $\pi_1(\ot\cG,(v,n))\sgr C(z^n)$.

    Conversely, take $g\in C(z^n)$ with reduced sound writing $g=g_0e_1g_1e_2g_2\dots g_{\ell-1}e_\ell g_\ell\in\pi_1(\cG,v)$, and call $q_i = \frac{\psi(e_i)}{\psi(\ol{e}_i)} \in \bbQ\setminus\{0\}$. Since $z^ng=gz^n$, by Britton's Lemma \ref{lem:sound-writing}, we have that $\psi(\ol{e}_i)\divides nq_1\dots q_{i-1}$ for $i=1,\dots,\ell$ and $ q_1\dots q_\ell = 1$. Thus we obtain a well-defined element $\ot g\in\pi_1(\ot\cG,(v,n))$ given by the reduced writing
    $$\ot g = g_0(e_1,\frac{n}{\psi(\ol{e}_1)})g_1(e_2,\frac{nq_1}{\psi(\ol{e}_2)})g_2\dots g_{\ell-1}(e_\ell,\frac{nq_1\dots q_{\ell-1}}{\psi(\ol{e}_\ell)})g_\ell$$
    and thus $C(z^n)\sgr\pi_1(\ot\cG,(v,n))$.
\end{proof}

The description of centralizers of elliptic elements in a graph of groups can be extended to the description of centralizers of hyperbolic elements. We record it in the following

\begin{prop}[centralizers of hyperbolic elements]
Let $G$ be a GBS group. Suppose that $g$ is a hyperbolic element and $g$ and $h$ commute in $G$. Then there exist a {\rm(}cyclically
reduced{\rm)} hyperbolic element $r\in G$, $w\in G$, and $b, c$ elliptic elements such that 
$$
g = \ol{w} (r^nb)w, \quad \quad h = \ol{w}(r^m c)w,
$$
where $n,m \in \mathbb Z$, $n\ne 0$, $r\in C(b)\cap C(c)$, and $[b,c]=1$.
\end{prop}

\section{Automorphisms of graphs of groups relative to vertices} \label{sec:relautos}

In this section, we consider a group acting on a tree, and we introduce the notion of \textit{outer automorphism relative to the vertices}, or shortly \textit{relative outer automorphisms}. These are outer automorphism that act as a conjugation on each vertex group, see \Cref{def:relauto-tree}. We observe that the group of relative outer automorphisms depends only on the deformation space of the tree.

\subsection{Relative endomorphisms of a graph of groups}

Let $\cG=(\Gamma,\{G_v\},\{G_e\},\{\psi_e\})$ be a graph of groups. Recall from Section \ref{sec:graphs-of-groups} that the universal group $\FG{\cG}$ is defined as the quotient of the free product $(*_{v\in V(\Gamma)}G_v)*F(E(\Gamma))$ by the relations
\[
    \ol{e}=e^{-1} \qquad\qquad \psi_{\ol{e}}(g)\cdot e=e\cdot\psi_e(g)
\]
for $e\in E(\Gamma)$ and $g\in G_e$. We denote by $\pi_1(\cG,v,u)$ the set of $(v,u)$-sound elements of $\FG{\cG}$ and by $\pi_1(\cG,v_0)=\pi_1(\cG,v_0,v_0)$.

\begin{defn}\label{def:relendo-universalgroup}
    A \textbf{relative endomorphism} of $\cG$ is a map $\Theta:\FG{\cG}\rar\FG{\cG}$ satisfying the following conditions:
    \begin{enumerate}
        \item $\Theta$ is a homomorphism.
        \item For all $v\in V(\Gamma)$ and $g\in G_v$, we have $\Theta(g)=g$.
        \item For all $e\in E(\Gamma)$, we have $\Theta(e)\in\pi_1(\cG,\iota(e),\tau(e))$.
    \end{enumerate}
\end{defn}

Since $\Theta$ is a homomorphism, in particular we must have that $\Theta(\ol{e})=\Theta(e)^{-1}$ and $\psi_{\ol{e}}(h)\Theta(e)=\Theta(e)\psi_e(h)$ for $e\in E(\Gamma)$ and $h\in G_e$. From the definition, it follows that if $x\in\pi_1(\cG,v,u)$ then $\Theta(x)\in\pi_1(\cG,v,u)$. We observe that the composition of relative endomorphisms is a relative endomorphism.

\begin{defn}\label{def:centralizer-subgroup}
    For $v\in V(\Gamma)$, define the \textbf{centralizer}
    \[
        C(G_v):=\{a\in\pi_1(\cG,v,v) : ag=ga \text{ for all } g\in G_v\}\sgr\pi_1(\cG,v,v).
    \]
\end{defn}

\begin{remark}
    Note that this is the centralizer in \textit{the whole fundamental group} of the graph of groups (and not just the centralizer inside $G_v$ itself). Also, centralizers of different vertices are subgroups of elements at different basepoints.
\end{remark}

\begin{defn}\label{def:relendo-equivalence}
    Two relative endomorphisms $\Theta,\Theta':\FG{\cG}\rar\FG{\cG}$ are called \textbf{equivalent} if there are elements $a_v\in C(G_v)$ for $v\in V(\Gamma)$ such that $\Theta'(e)=a_{\iota(e)}\Theta(e)\ol{a}_{\tau(e)}$ for all $e\in E(\Gamma)$.
\end{defn}

This defines an equivalence relation on the set of relative endomorphisms. A relative endomorphism $\Theta$ preserves the centralizers, i.e. $\Theta(C(G_v))\sgr C(G_v)$ for all $v\in V(\Gamma)$. It follows that the composition of two relative endomorphisms is well-defined up to equivalence.

\begin{lem}\label{lem:true-inverse}
    For a relative endomorphism $\Theta:\FG{\cG}\rar\FG{\cG}$, the following are equivalent:
    \begin{enumerate}
        \item There is a relative endomorphism $\Theta'$ such that $\Theta\circ\Theta'$ and $\Theta'\circ\Theta$ are equivalent to $\id$.
        \item There is a relative endomorphism $\Theta'$ such that $\Theta\circ\Theta'$ and $\Theta'\circ\Theta$ are equal to $\id$.
    \end{enumerate}
\end{lem}
\begin{proof}
    It is enough to prove that if $\Theta$ is equivalent to the identity, then there is $\Theta'$ such that $\Theta\circ\Theta'=\Theta'\circ\Theta=\id$. Suppose that $\Theta$ is equivalent to the identity, so that there are constants $a_v\in C(G_v)$ for $v\in V(\Gamma)$ such that $\Theta(e)=a_{\iota(e)}e\ol{a}_{\tau(e)}$. For $v\in V(\Gamma)$ we write $a_v=g_0e_1g_1e_2g_2\dots g_{\ell-1}e_\ell g_\ell$ and we get
\begin{gather}\notag
\begin{split}
\Theta(a_v)= & g_0(a_{\iota(e_1)}e_1\ol{a}_{\tau(e_1)})g_1(a_{\iota(e_2)}e_2\ol{a}_{\tau(e_2)})g_2\dots g_{\ell-1}(a_{\iota(e_\ell)}e_\ell \ol{a}_{\tau(e_\ell)})g_\ell= \\
        = & a_vg_0e_1g_1e_2g_2\dots g_{\ell-1}e_\ell g_\ell \ol{a}_v=
    a_va_v\ol{a}_v=
    a_v.
\end{split}
\end{gather}
    We define $\Theta'$ given by $\Theta'(e)=\ol{a}_{\iota(e)}ea_{\tau(e)}$. In the same way as above, we have that $\Theta'(a_v)=a_v$ for all $v\in V(\Gamma)$. Now we have that $\Theta'(\Theta(e))=\Theta'(a_{\iota(e)}e\ol{a}_{\tau(e)})=a_{\iota(e)}\ol{a}_{\iota(e)}ea_{\tau(e)}\ol{a}_{\tau(e)}=e$ and similarly $\Theta(\Theta'(e))=e$, as desired.
\end{proof}

\subsection{Restriction of relative endomorphisms to the fundamental group}

Let $\cG=(\Gamma,\{G_v\},\{G_e\},\{\psi_e\})$ be a graph of groups and let $v_0\in V(\Gamma)$ be a basepoint.

\begin{defn}\label{def:relendo-fundamentalgroup}
    A \textbf{relative endomorphism} of $\pi_1(\cG,v_0)$ is a homomorphism $\theta:\pi_1(\cG,v_0)\rar\pi_1(\cG,v_0)$ such that, for all $v\in V(\Gamma)$ and $x\in\pi_1(\cG,v_0,v)$, there is $y\in\pi_1(\cG,v_0,v)$ such that $\theta(xg\ol{x})=yg\ol{y}$ for all $g\in G_v$.
\end{defn}

The element $y\in\pi_1(\cG,v_0,v)$ of the above Definition \ref{def:relendo-fundamentalgroup} is not unique in general; the other elements $y'$ satisfying the same property are exactly the elements of the coset $yC(G_v)$. Notice that the composition of relative endomorphism is a relative endomorphism.

\begin{prop}\label{prop:relendo-FG-pi1}
    For $v_0\in V(\Gamma)$, there is a bijection between the following two sets:
    \begin{enumerate}
        \item Relative endomorphisms $\Theta:\FG{\cG}\rar\FG{\cG}$ up to equivalence.
        \item Relative endomorphisms $\theta:\pi_1(\cG,v_0)\rar\pi_1(\cG,v_0)$ up to post-composing with conjugations.
    \end{enumerate}
    The bijection preserves compositions.
\end{prop}
\begin{proof}
    Given a relative endomorphism $\Theta:\FG{\cG}\rar\FG{\cG}$ we can consider its restriction $\theta:\pi_1(\cG,v_0)\rar\pi_1(\cG,v_0)$, and this is a relative endomorphism. Changing $\Theta$ by equivalence changes $\theta$ by post-composing with a conjugation (by an element of $C(G_{v_0})$). The restriction also behaves well with respect to composition.

    Conversely, given a relative endomorphism $\theta:\pi_1(\cG,v_0)\rar\pi_1(\cG,v_0)$, we proceed as follows. For every $v\in V(\Gamma)$ we choose $x_v\in\pi_1(\cG,v_0,v)$ and we find $y_v\in\pi_1(\cG,v_0,v)$ such that $\theta(x_vg\ol{x}_v)=y_vg\ol{y}_v$ for all $g\in G_v$; we now define the relative endomorphism $\Theta:\FG{\cG}\rar\FG{\cG}$ by setting $\Theta(e)=\ol{y}_{\iota(e)}\theta(x_{\iota(e)}e\ol{x}_{\tau(e)})y_{\tau(e)}$ for $e\in E(\Gamma)$. This is a well-defined homomorphism since for $h\in G_e$ we have
    \begin{align*}
        \psi_{\ol{e}}(h)\Theta(e)=
        \ol{y}_{\iota(e)}(y_{\iota(e)}\psi_{\ol{e}}(h)\ol{y}_{\iota(e)})\theta(x_{\iota(e)}e\ol{x}_{\tau(e)})y_{\tau(e)}=
        \ol{y}_{\iota(e)}\theta(x_{\iota(e)}\psi_{\ol{e}}(h)\ol{x}_{\iota(e)})\theta(x_{\iota(e)}e\ol{x}_{\tau(e)})y_{\tau(e)}=
        \\
        \ol{y}_{\iota(e)}\theta(x_{\iota(e)}e\ol{x}_{\tau(e)})\theta(x_{\tau(e)}\psi_{e}(h)\ol{x}_{\tau(e)})y_{\tau(e)}=
        \ol{y}_{\iota(e)}\theta(x_{\iota(e)}e\ol{x}_{\tau(e)})(y_{\tau(e)}\psi_{e}(h)\ol{y}_{\tau(e)})y_{\tau(e)}=
        \Theta(e)\psi_{e}(h).
    \end{align*}
    Fixed the elements $x_v\in\pi_1(\cG,v_0,v)$, different choice of elements $y_v'\in y_vC(G_v)$ gives equivalent relative endomorphisms $\Theta$. A different choice of elements $x_v'\in\pi_1(\cG,v_0,v)$ gives the equivalent relative endomorphisms $\Theta$, since we can just take $y_v'=\theta(x_v'\ol{x}_v)y_v$. Composing $\theta$ with a conjugation (both before and after) does not affect $\Theta$.

    These two correspondences are inverses of each other. In particular, the bijection behaves well in both directions with respect to composition. The statement  follows.
\end{proof}

\subsection{Relative outer automorphisms}\label{sec:relauto-tree}

\begin{defn}\label{def:relauto-tree}
    Let $G$ be a group acting on a tree $T$. Define the subgroup $\outrel{G}{T}\sgr\out{G}$ given by the {\rm(}equivalence classes of{\rm)} automorphisms $\theta:G\rar G$ satisfying the following property: for all $v\in V(T)$, there is $x\in G$ such that $\theta(g)=xg\ol{x}$ for all $g\in G_v$.
\end{defn}

Let $G$ be a group acting on a tree $T$, let $\cG=(\Gamma,\{G_v\},\{G_e\},\{\psi_e\})$ be the corresponding graph of groups, and let $v_0\in V(\Gamma)$ be a basepoint. Then we have that an automorphism $\theta:G\rar G$ belongs to $\outrel{G}{T}$ if and only if the corresponding automorphism $\theta:\pi_1(\cG,v_0)\rar\pi_1(\cG,v_0)$ is a relative endomorphism. In particular, we have a natural bijection between the following groups:
\begin{enumerate}
    \item $\outrel{G}{T}$.
    \item The subgroup of $\out{\pi_1(\cG,v_0)}$ given by the (equivalence classes of) relative endomorphisms which are also automorphisms.
    \item The group of relative endomorphisms $\Theta:\FG{\cG}\rar\FG{\cG}$, considered up to equivalence, that admit an inverse.
\end{enumerate}

\begin{lem}\label{lem:relauto-domination}
    Let $G$ be a group acting on two trees $T,T'$, and suppose that $T$ dominates $T'$. Then $\outrel{G}{T'}\sgr\outrel{G}{T}$.
\end{lem}
\begin{proof}
    Let $[\theta]\in\outrel{G}{T'}$. For $v\in V(T)$ we can find $v'\in V(T')$ such that $G_v\sgr G_{v'}$. By definition, there is $x\in G$ such that $\theta(g)=xg\ol{x}$ for all $g\in G_{v'}$. In particular, $\theta(g)=xg\ol{x}$ for all $g\in G_v$.
\end{proof}

In particular, the group $\outrel{G}{T}$ depends only on the deformation space of the tree $T$.

\section{Comparing automorphism groups}\label{sec:comparing-automorphism-groups}

In his influential paper \cite{Sel97}, Sela introduced a new tool, the canonical JSJ decomposition of torsion-free hyperbolic groups, and showed that it can be used to describe the outer automorphism group in terms of the splitting (as the fundamental group of a graph of groups with cyclic edge groups). More precisely, he showed that the outer automorphism group is (virtually) 
generated by Dehn twists along edges of the splitting, together with outer automorphisms of the vertex groups, which in that case correspond to mapping class groups of punctured surfaces. These results were generalized by Levitt in \cite{Lev05} to simplify and extend Sela's results for torsion-free hyperbolic groups, and later by Guirardel-Levitt in \cite{GL15} to relatively hyperbolic groups.

In this section, we show that the outer automorphism group of the fundamental group of a graph of groups with cyclic edge groups and whose vertex groups are torsion-free and have cyclic centralizers, can also be described using the JSJ decomposition. In this case, the outer automorphism group is (virtually) generated by (some) outer automorphisms of the vertex groups, the outer automorphism groups of canonically associated generalized Baumslag-Solitar groups, and generalized twists --- partial conjugations by elements in the centralizer of some elliptic elements, see Theorem \ref{thm:description-outrel}. In particular, we describe the outer automorphism groups of one-ended fundamental groups of graphs of groups of torsion-free hyperbolic groups with cyclic edge groups.

\begin{thm*}[see Corollary \ref{cor:graph_hyperbolic_group}]
Let $\cG$ be a graph of groups with torsion-free hyperbolic vertex groups and infinite cyclic edge groups. Suppose that $G=\pi_1(\cG)$ is freely indecomposable. Then there are exact sequences
     \[
    \begin{array}{c}
    1 \to \outr{G} \to \outz{G} \to \prod MCG(\Sigma_v)\times (\hbox{finite group})\times (\hbox{fg abelian});\\[0.2cm]
        1\rar Z \rar\prod_{\rho\in\cR}C(\rho) \xrightarrow{\tau} \outr{G}\xrightarrow{\sigma}\outr{G_1}\times\outr{G_2}\times\dots \times\outr{G_s}\times(D_\infty)^t \to 1
    \end{array}
    \]
    for some finite set of elliptic elements $\cR\subseteq G$, for some finitely generated free abelian group $Z$, for some integers $s,t\ge0$, and for some GBSs $G_1,\dots,G_s$ {\rm(}here $D_\infty$ is the infinite dihedral group{\rm)}.
\end{thm*}

From the combination theorem of Bestvina and Feighn, see \cite{BF92}, we have that all the associated GBSs $G_i$ are cyclic, {\rm(}and so are the centralizers $C(\rho)${\rm)} if and only if $G$ is hyperbolic. In this case, since there are no solvable Baumslag-Solitar groups, there are no abelian groups in the first exact sequence; and $\outr {G}$ fits into an exact sequence with trivial quotient and abelian kernel, that is $\outr{G}$ is finitely generated abelian. We therefore recover the result of Sela for one-ended torsion-free hyperbolic groups, see Theorem \ref{thm:sela}.

\subsection{Decomposition of the outer automorphism group}

Let $G$ be a finitely presented torsion-free group, and suppose that $G$ is freely indecomposable. By Theorem \ref{thm:JSJ-existence} there exists a JSJ decomposition of $G$ over the family of its cyclic subgroups. By \cite{ACK-iso1}, we can take a \textit{totally reduced} JSJ decomposition (i.e. if two vertices have stabilizers one contained in the other, then they belong to the same orbit).

\begin{lem}\label{lem:bijection-vertices}
    Let $J,J'$ be two totally reduced JSJ decompositions for $G$ over the family of its cyclic subgroups, and let $\Gamma,\Gamma'$ be the quotient graphs. Then there is a bijection $b:V(\Gamma')\rar V(\Gamma)$ satisfying the following:
    \begin{enumerate}
        \item For all $v\in V(J)$ there is $v'\in V(J')$ such that $G_v\sgr G_{v'}$, and all such $v'$ belong to the same orbit $b^{-1}(v/G)$.
        \item For all $v'\in V(J')$ there is $v\in V(J)$ such that $G_{v'}\sgr G_v$, and all such $v$ belong to the same orbit $b(v'/G)$.
    \end{enumerate}
\end{lem}
\begin{proof}
    Take a vertex $v\in V(J)$, take a vertex $v'\in V(J')$ such that $G_v\sgr G_{v'}$ (which exists since $J,J'$ are JSJ decompositions), take a vertex $v''\in V(J)$ such that $G_{v'}\sgr G_{v''}$ (which exists since $J,J'$ are JSJ decompositions). This implies that $G_v\sgr G_{v''}$, but since $J$ is totally reduced, this forces $v,v''$ to belong to the same orbit of vertices of $J$. The statement  follows.
\end{proof}

Let $J$ be a totally reduced JSJ decompositions for $G$ over the family of its cyclic subgroups, and let $\cG=(\Gamma,\{G_v\},\{G_e\},\{\psi_e\})$ be the corresponding graph of groups. An automorphism $\theta:G\rar G$ allows us to construct another totally reduced JSJ decomposition $J'$, given by the same tree $J$ with the action $g\cdot'p:=\theta^{-1}(g)\cdot p$. Notice that the two actions have the same orbits. If $v\in V(J)$ we denote by $G_v$ the stabilizer under the action $\cdot$ and with $G_v'=\theta(G_v)$ the stabilizer under the action $\cdot'$. In particular, by Lemma \ref{lem:bijection-vertices} we obtain a bijection $\perm{\theta}:V(\Gamma)\rar V(\Gamma)$. 
A composition of automorphisms gives a composition of bijections, and a conjugation automorphism gives the identity. In particular, we obtain a homomorphism $\Perm:\out{G}\rar\Sigma$, where $\Sigma$ is the finite group of permutations of the finite set $V(\Gamma)$.

\begin{defn}\label{def:outz}
    Define the finite index normal subgroup $\outz{G}\nor\out{G}$ given by the kernel of the map $\Perm$, so that we have an exact sequence 
    \[
        1\rar\outz{G}\rar\out{G}\rar\Sigma.
    \]
\end{defn}

\begin{remark}
     $\outz{G}$ does not depend on the chosen totally reduced JSJ decomposition $J$.
\end{remark}

Let $[\theta]\in\outz{G}$ where $\theta:G\rar G$ is an automorphism. Chosen a basepoint $v_0\in\Gamma$, this can be considered as an automorphism $\theta:\pi_1(\cG,v_0)\rar\pi_1(\cG,v_0)$.

Take a vertex $v\in V(\Gamma)$ with $G_v\not\cong\bbZ$. Then we can find $x\in\pi_1(\cG,v_0,v)$ such that $\theta(G_v)\sgr xG_v\ol{x}$; the fact that $G_v\not\cong\bbZ$ forces the equality $\theta(G_v)=xG_v\ol{x}$. In particular, we obtain an automorphism $\theta_{v,x}:G_v\rar G_v$ given by $g\mapsto \ol{x}\theta(g)x$. A different choice of $x'\in\pi_1(\cG,v_0,v)$ must satisfy $x'\in xG_v$, and thus the automorphism $\theta_{v,x'}$ differs from $\theta_{v,x}$ by post-composition with a conjugation. In particular, we obtain a well-defined element $[\theta_v]\in\out{G_v}$. This defines a homomorphism $\eta_v:\outz{G}\rar\out{G_v}$ given by $[\theta]\mapsto[\theta_v]$.

Take a vertex $v\in V(\Gamma)$ with $G_v\cong\bbZ$; let $z$ be a generator for $G_v$. We consider the set of primes $\cP(v)=\{p$ prime $: z$ is conjugate to $z^a$ for some $a$ multiple of $p\}$ and we define the finitely generated abelian group $P_v\sgr\bbQ\setminus\{0\}$ of the rational numbers such that their factorization contains only prime numbers in $\cP(v)$ (possibly with negative exponent). We define the finitely generated abelian group $Q_v\sgr P_v$ as the one generated by the integers $a\in\bbZ\setminus\{0\}$ such that $z$ is conjugate to $z^a$. We can find $x\in\pi_1(\cG,v_0,v)$ such that $\theta(z)=xz^b\ol{x}$ for some integer $b\in P_v$; moreover, a different choice of element $x'\in\pi_1(\cG,v_0,v)$ and integer $b'\in P_v$ with $\theta(z)=x'z^{b'}\ol{x}'$ must satisfy $b'\ol{b}\in Q_v$. In particular, we obtain a well-defined element $b/Q_v\in P_v/Q_v$. This defines a homomorphism $\eta_v:\outz{G}\rar P_v/Q_v$ given by $[\theta]\mapsto b/Q_v$.

\begin{remark}
    We point out that a finite set of generators for $Q_v$ can be computed algorithmically. In fact, such a set of generators is given (up to passing from the additive to the multiplicative notation) by the elements $\bw_1,\dots,\bw_m$ given by \cite[Proposition 3.9]{ACK-iso1}. Similarly, the quotient $P_v/Q_v$ can be computed algorithmically.
\end{remark}

We define 
$$\eta = \prod_{v\in V(\Gamma)}\eta_v:\outz{G}\to
\prod_{\underset{G_v\not\cong\bbZ}{v\in V(\Gamma)}}
\out{G_v}\times
\prod_{\underset{G_v\cong\bbZ}{v\in V(\Gamma)}}
P_v/Q_v.$$
It is immediate from the definitions that $[\theta]\in\outz{G}$ belongs to all the kernels of the maps defined above if and only if $\theta:\pi_1(\cG,v_0)\rar\pi_1(\cG,v_0)$ is a relative automorphism (according to Definition \ref{def:relendo-fundamentalgroup}). 

\begin{prop}\label{prop:outr}
    We have an exact sequence
    \[
        1\rar\outr{G}\rar\outz{G} \xrightarrow{\eta}
        \prod_{\underset{G_v\not\cong\bbZ}{v\in V(\Gamma)}}
        \out{G_v}\times
        \prod_{\underset{G_v\cong\bbZ}{v\in V(\Gamma)}}
        P_v/Q_v,
    \]
    where $\outr{G}=\outrel{G}{J}$ for any JSJ decomposition $J$ of $G$ over its cyclic subgroups.
\end{prop}
\begin{remark}
    In Proposition \ref{prop:outr} above, the quotients $P_v/Q_v$ are finitely generated abelian groups.
\end{remark}
\begin{proof}
    Immediately follows from the above discussion.
\end{proof}

\begin{remark}
We can define the mapping class group $MCG(G_v) < \out{G_v}$ by restricting to automorphisms which act on each edge group $G_e$ as conjugation by some
$g_e \in G_v$, for $e\in E(\Gamma)$ such that $\tau(e)=v$. In \cite[Proposition 2.1]{Lev05}, Levitt shows that $\prod_{v\in V(\Gamma)} MCG(G_v)$ is contained in the image of $\eta$ (which coincides with the image of the subgroup of outer automorphisms that fixes $\Gamma$). Furthermore, if $\out{G_e}$ is finite, then $\prod_{v\in V(\Gamma)} MCG(G_v)$ has finite index in the image.
\end{remark}

\subsection{Choice of roots}

\begin{defn}
    A group $H$ has \textbf{cyclic centralizers} if, for every $h\in H$, its centralizer $\{c\in H : ch=hc\}$ is cyclic.
\end{defn}

An element $r\in H$ is called \textbf{root} if it is not a proper power of another element (i.e. if $r\not=s^m$ for all $s\in H$ and $m\ge2$). If $H$ is a torsion-free group with cyclic centralizers, then every element $h\in H\setminus\{1\}$ has a unique root, i.e. there is a unique root element $r_h\in H$ such that $h=r_h^k$ for some $k\ge1$. The element $r_h$ is one of the two generators of the centralizer $\{c\in H : ch=hc\}\cong\bbZ$.

\begin{remark}
    Suppose that $H$ is a torsion-free group with cyclic centralizers, and let $h\in H\setminus\{1\}$. If $h^i$ is conjugate to $h^j$ for some $i,j\in\bbZ$, then $i=j$, see \cite[Lemma 8.21]{ACK-iso1}. In particular, the centralizer of $h$ coincides with the normalizer of the cyclic subgroup $\gen{h}$.
    \end{remark}

Let $\fF$ be a family of finitely presented torsion-free groups with cyclic centralizers. We consider a graph of groups with vertex groups in $\fF$ and infinite cyclic edge groups. For every edge, we can take the generator of such edge, and look at its image in the adjacent vertex group: we want to choose, inside the vertex group, a root for such element. We do not want this list of roots to contain repetitions: if two edge groups are glued on (powers of) the same root, then we want to list the root only once. Moreover, since edge inclusions can be changed by conjugation without changing the corresponding Bass-Serre tree, we want conjugate roots to be considered the same. This is all summarized in the following definition.

\begin{defn}\label{def:choice-of-roots}
    Let $\cG=(\Gamma,\{G_v\}_{v\in V(\Gamma)},\{G_e\}_{e\in E(\Gamma)},\{\psi_e\}_{e\in E(\Gamma)})$ be a graph of groups with vertex groups in $\fF$ and edge groups isomorphic to $\bbZ$.
    \begin{enumerate}
        \item A \textbf{choice of cyclic roots} is a finite set $\cyclicR{\cG}$, minimal by inclusion, satisfying the following property: for all $v\in V(\Gamma)$ with $G_v\cong\bbZ$, the set $\cyclicR{\cG}$ contains exactly one generator of $G_v$.
        \item A \textbf{choice of non-cyclic roots} is a finite set $\notcyclicR{\cG}$, minimal by inclusion, satisfying the following property: for all $e\in E(\Gamma)$ with $G_{\tau(e)}\not\cong\bbZ$, there is a unique $\rho\in\notcyclicR{\cG}$ such that the root of $\psi(e)$ is conjugate to $\rho^{\pm1}$ in $G_{\tau(e)}$.
        \item A \textbf{choice of roots} is a union $\cR(\cG)=\notcyclicR{\cG}\cup\cyclicR{\cG}$ for some choices of non-cyclic and cyclic roots.
    \end{enumerate}
    These are uniquely determined up to taking inverses and conjugates of the roots in their respective vertex groups.
\end{defn}

\subsection{The GBS core of a graph of groups}\label{sec:GBS-core}

Let $\fF$ be a family of finitely presented torsion-free groups with cyclic centralizers. Let $\cG=(\Gamma,\{G_v\},\{G_e\},\{\psi_e\})$ be a graph of groups with vertex groups in $\fF$ and infinite cyclic edge groups. Consider a choice of roots $\cR(\cG)$ (see Definition \ref{def:choice-of-roots}). Up to changing the inclusion map $\psi_e:G_e\rar G_{\tau(e)}$ by conjugation, we can assume that the image of $\psi_e:G_e\rar G_{\tau(e)}$ is contained in $\gen{\rho_e}$ for some (uniquely determined) $\rho_e\in\cR(\cG)$.

\begin{defn}
    Define the \textbf{GBS core} $\cD=(\Delta,\{H_v\}_{v\in V(\Delta)},\{H_e\}_{e\in E(\Delta)},\{\phi_e\}_{e\in E(\Delta)})$ of $\cG$ as follows {\rm(}see {\rm Figure \ref{fig:reduction-GBS})}:
    \begin{enumerate}
        \item $\Delta$ is a graph. We set $V(\Delta)=\cR(\cG)$. We set $E(\Delta)=E(\Gamma)$ with the same reverse map.
        \item For $e\in E(\Delta)$, we set $\tau_{\Delta}(e)=\rho_e$.
        \item For $\rho\in V(\Delta)$, define the group $H_v=\gen{\rho}\cong\bbZ$. For $e\in E(\Delta)$, define the group $H_e=G_e\cong\bbZ$.
        \item For $e\in E(\Delta)$, define the homomorphism $\phi_e:H_e\rar H_{\tau(e)}$ by setting $\phi_e=\psi_e$.
    \end{enumerate}
   
\end{defn}

The quadruple $\cD$ satisfies all the hypotheses for being a graph of groups, except that the finite graph $\Delta$ may be disconnected. We write $\Delta=\Delta_1\sqcup\dots \sqcup\Delta_k$ as the disjoint union of its connected components, for $i=1,\dots,k$, and we have that every $\Delta_i$ defines a GBS graph of groups $\cD_i$. In what follows, we will say that $\cD=\cD_1\sqcup\dots\sqcup\cD_k$ is the GBS core of $\cG$, meaning that $\cD_i$ are the GBS graphs of groups induced by the connected components of $\cD$.

\begin{figure}[H]
	\centering
	\begin{tikzpicture}[scale=1]
			\begin{scope}[shift={(0,0)}]
			\node (p1) at (0.2,0) {.};
			\node (p2) at (-.3,.3) {.};
			\node () at (0.1,-0.2) {$\rho_1$};
			\node () at (-0.4,.1) {$\rho_2$};
			\draw (0,0) circle [radius=0.6]; 
			\draw (3,0) circle [radius=0.6]; 
			\node (p5) at (2.8,0) {.};
			\node () at (2.8,-0.2) {$\rho_5$};
		
	   	\draw (1.5,2) circle [radius=0.6]; 
            \node (p3) at (1.5,1.7) {.};
	   	\node () at (1.3,1.8) {$\rho_3$};
            \node (p4) at (1.7,2.2) {.}; 	
	   	\node () at (1.5,2.3) {$\rho_4$};
            \draw[-] (p2) to (p3);
            \draw[-,out=10,in=170,looseness=1] (p1) to (p5);
		\draw[-,out=-10,in=-170,looseness=1] (p1) to (p5);
	   \draw[-,out=-30,in=-20,looseness=10] (p3) to (p4);	
		\draw[-,out=10,in=45,looseness=50] (p4) to (p4);	
		\node () at (1.5,-1.5) {$\Gamma$ (with roots in evidence)};	
		\end{scope}

	\begin{scope}[shift={(6,0)}]	
		\node (p1) at (0.2,0) {.};
		\node (p2) at (-.3,.3) {.};	
		\node () at (0.1,-0.2) {$\rho_1$};
		\node () at (-0.4,.1) {$\rho_2$};
  \node (p5) at (2.8,0) {.};
	\node () at (2.8,-0.2) {$\rho_5$};
		\node (p3) at (1.3,1.5) {.};
		\node () at (1.1,1.6) {$\rho_3$};
		\node (p4) at (1.9,2.4) {.}; 	
		\node () at (1.7,2.5) {$\rho_4$};
		\draw[-] (p2) to (p3);
		\draw[-] (p3) to (p4);
		\draw[-,out=10,in=170,looseness=1] (p1) to (p5);
		\draw[-,out=-10,in=-170,looseness=1] (p1) to (p5);
		\draw[-,out=10,in=45,looseness=50] (p4) to (p4);
		\node () at (1.5,-1.5) {$\Delta$};
	\end{scope}
	
	\begin{scope}[shift={(-5,0)}]
		\node (p1) at (0.2,0) {.};
		\node (p5) at (2.8,0) {.};
		\node (p4) at (1.5,1.5) {.}; 	
  \draw[-] (p1) to (p4);
        \draw[-,out=10,in=170,looseness=1] (p1) to (p5);
		\draw[-,out=-10,in=-170,looseness=1] (p1) to (p5);
			\draw[-,out=10,in=45,looseness=50] (p4) to (p4);
	\draw[-,out=90,in=125,looseness=50] (p4) to (p4);
		\node () at (1.5,-1.5) {$\Gamma$};
		\end{scope}
	\end{tikzpicture}
	\caption{The (possibly disconnected) GBS graph associated with a graph of groups with cyclic edges.}
	\label{fig:reduction-GBS}
\end{figure}
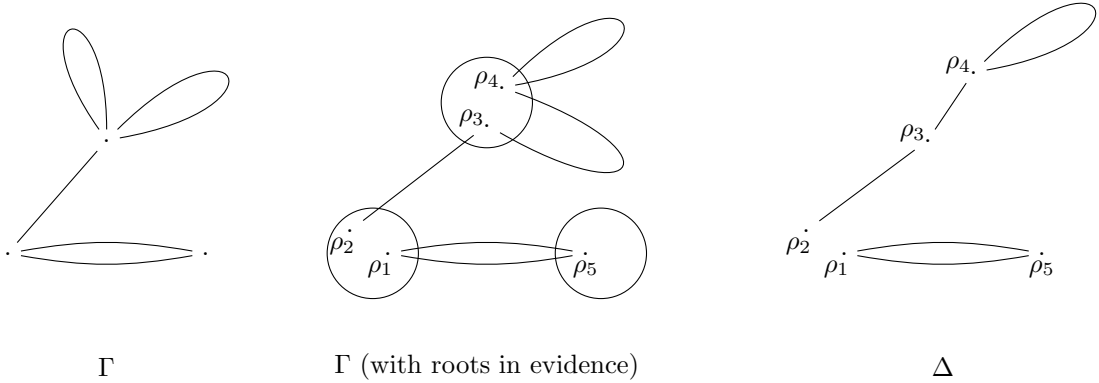

\begin{remark}\label{rem:embedding_GBS_in_group}
For every $i=1,\dots,k$, we have a natural homomorphism $\FG{\cD_i}\rar\FG{\cG}$: by definition, the set of generators for $\FG{\cD_i}$ is a subset of the set of generators for $\FG{\cG}$, and relations of $\FG{\cD_i}$ are also relations in $\FG{\cG_i}$. It is immediate to check that such inclusion sends sound elements to sound elements. Given $\rho\in V(\Delta_i)\subseteq\cR(\cG)$ with $\rho\in G_v$ for $v\in V(\Gamma)$, we obtain a homomorphism $\pi_1(\cD_i,\rho)\rar\pi_1(\cG,v)$. Using Britton's Lemma \ref{lem:sound-writing}, we obtain that this homomorphism $\pi_1(\cD_i,\rho)\rar\pi_1(\cG,v)$ is injective, meaning that $\pi_1(\cD_i,\rho)$ is a subgroup of $\pi_1(\cG,v)$.    
\end{remark}

In what follows, we will need the following Proposition \ref{prop:comparison-of-centralizers}, stating that the centralizers of the roots in $\cG$ are actually contained in $\cD$. Recall from Definition \ref{def:centralizer-element} the notion of centralizer of an element in a graph of groups.

\begin{prop}[Comparison of centralizers]\label{prop:comparison-of-centralizers}
    Let $\cD=\cD_1\sqcup\dots\sqcup\cD_k$ be the GBS core of $\cG$. Then, for every $\rho\in\cR(\cG)$, we have $C_\cG(\rho)\sgr\pi_1(\cD_1,\rho)$, where $\cD_1$ is the connected component of $\cD$ containing $\rho$, and in particular $C_\cG(\rho)=C_{\cD_1}(\rho)$.
\end{prop}
\begin{proof}
    Let $\rho\in G_v$ for some $v\in V(\Gamma)$. Let $g\in C_\cG(\rho)$ and take a reduced $(v,v)$-sound writing $g=g_0e_1g_1\dots g_{\ell-1}e_\ell g_\ell$. By Britton's lemma, since $g\rho=\rho g$, we must have that $e_i,\ol{e}_{i+1}$ are glued to a common root $\rho_i$ and $g_i$ must be a power of $\rho_i$. It follows that $g$ is written in the generators of $\cD_1$, and thus $g\in\pi_1(\cD_1,\rho)$. The statement  follows.
\end{proof}

\subsection{Relative outer automorphisms induced by centralizers}\label{sec:auto-centralizers}

Let $\fF$ be a family of finitely presented torsion-free groups with cyclic centralizers. Let $\cG=(\Gamma,\{G_v\},\{G_e\},\{\psi_e\})$ be a graph of groups with vertex groups in $\fF$ and with infinite cyclic edge groups. Consider a choice of roots $\cR(\cG)$ (see Definition \ref{def:choice-of-roots}). Up to changing the inclusion map $\psi_e:G_e\rar G_{\tau(e)}$ by conjugation, we can assume that the image of $\psi_e:G_e\rar G_{\tau(e)}$ is contained in $\gen{\rho_e}$ for some (uniquely determined) $\rho_e\in\cR(\cG)$.

For every root $\rho\in\cR(\cG)$ choose an element $c_\rho\in C_\cG(\rho)$ of its centralizer. Define the relative endomorphism $\Theta:\FG{\cG}\rar\FG{\cG}$ by setting $\Theta(e)=c_{\rho_{\ol{e}}}e\ol{c}_{\rho_e}$ for $e\in E(\Gamma)$. This means that for every root $\rho$, we are twisting all the edges connected to that root by the same element $c_\rho$. 

\begin{lem}\label{lem:identity-on-centralizers}
    The relative endomorphism $\Theta$ induces the identity in the centralizer $C_\cG(g)$ for every $g\in G_v$ with $v\in V(\Gamma)$.
\end{lem}
\begin{proof}
    If $g\in G_v$ is not conjugated in $G_v$ to any root in $\cR(\cG)$, then $C(g)$ is contained in $G_v$ and we are done. For $\rho\in\cR(\cG)\cap G_v$, an arbitrary element of the centralizer $C(\rho)$ will have a sound writing $g_0e_1g_1\dots g_{\ell-1}e_\ell g_\ell$. Since the centralizers of the roots are cyclic in the vertex groups, this implies that $g_i$ must be a power of $\rho_{e_i}=\rho_{\ol{e}_{i+1}}$. Using the definition of $\Theta$, it follows that $\Theta(\dots e_ig_i e_{i+1}\dots)=\dots e_i\ol{c}_{\rho_{e_i}}g_ic_{\rho_{\ol{e}_{i+1}}}e_{i+1}\ldots =\dots e_ig_i e_{i+1}\dots$ and thus $\Theta$ induces the identity in $C(\rho)$, as desired.
\end{proof}

It follows from Lemma \ref{lem:identity-on-centralizers} that if the relative endomorphism $\Theta$ is defined by the elements $c_\rho$ for $\rho\in\cR(\cG)$, and the relative endomorphism $\Theta'$ is defined by the elements $c_\rho'$ for $\rho\in\cR(\cG)$, then their composition $\Theta'\circ\Theta$ coincides with the relative endomorphism $\Theta''$ defined by the elements $c_\rho''=c_\rho'c_\rho$ for $\rho\in\cR(\cG)$. In particular, the relative endomorphism $\Theta$ admits an inverse. Note also that, for every root $\rho\in\cyclicR{\cG}$, changing the choice of $c_\rho$ will only change $\Theta$ by equivalence, and thus we obtain a well-defined homomorphism
\begin{equation}\label{eq:homomorphism-tau}
\tau:\prod_{\rho\in\notcyclicR{\cG}}C(\rho) \rar \outrel{G}{T}    
\end{equation}
where $G=\pi_1(\cG)$ and $T$ is the tree associated to the graph of groups.

\begin{remark}
    The above map is not injective in general. For example, consider a graph of groups $\cG$ with two vertices $v,w$ with groups $G_v,G_w\not\cong\bbZ$ with cyclic centralizers; the group $G_v$ contains a root element $\rho$ and the group $G_w$ contains a root element $\sigma$. The graph of groups $\cG$ has one edge $e$ with $\iota(e)=v$ and $\tau(e)=w$; the stabilizer $G_e$ is isomorphic to $\bbZ$, and the generator which is glued to $\rho$ on one side and to $\sigma^2$ on the other side. Then we can consider the elements $c_\rho=e\sigma\ol{e}\in C(\rho)$ and $c_\sigma=\sigma\in C(\sigma)$: these define the relative endomorphism $\Theta$ with $\Theta(e)=c_\rho e\ol{c}_\sigma=e\sigma\ol{e}e\ol{\sigma}=e$ which is the identity.
\end{remark}

\begin{prop}\label{prop:kernel-free-abelian}
    The kernel of the above map $\tau$ is a finitely generated free abelian group.
\end{prop}
\begin{proof}
    Consider a choice of elements $c_\rho\in C(\rho)$ for $\rho\in\notcyclicR{\cG}$ (and set $c_\rho=1$ for $\rho\in\cyclicR{\cG}$) and take the corresponding relative endomorphism $\Theta$, so that $\Theta(e)=c_{\rho_{\ol{e}}}e\ol{c}_{\rho_e}$. Suppose that $\Theta$ is equivalent to the identity. This means that there are elements $a_v\in C(G_v)$ for $v\in V(\Gamma)$ satisfying the following property: for all $e\in E(\Gamma)$ we have $c_{\rho_{\ol{e}}} e \ol{c}_{\rho_e} = \Theta(e) = a_{\iota(e)} e \ol{a}_{\tau(e)}$. Note that, whenever $G_v\not\cong\bbZ$, the centralizer $C(G_v)$ is trivial, and thus $a_v=1$.
    
    Now fix $v\in V(\Gamma)$ with $G_v\not\cong\bbZ$ and fix a root $\rho\in\notcyclicR{\cG}\cap G_v$. Suppose that
    \begin{equation}\label{eq:elements}
    x=\rho^{r_0} e_1 \rho_1^{r_1} e_2 \rho_2^{r_2} \dots \rho_{\ell-1}^{r_{\ell-1}} e_\ell \rho^{r_\ell}\in\pi_1(\cG,v)
    \end{equation}
    is a reduced sound writing, with $\rho_{e_i}=\rho_{\ol{e}_{i+1}}=\rho_i\in\cR(\cG)$ and $r_i\in\bbZ$ for $i=1,\dots,\ell-1$, and with $\rho_{e_\ell}=\rho_{\ol{e}_1}=\rho$. Then we must have $x = a_v x \ol{a}_v = \Theta(x) = c_\rho x \ol{c}_\rho$ and thus $c_\rho\in C(x)$.

    By the Bass-Serre correspondence, the graph of groups $\cG$ corresponds to an action of $\pi_1(\cG,v)$ on a tree $T$. 
    If $c_\rho\in C(x)$ and $x$ acts on $T$ hyperbolically with axis $L$, then $xc_\rho(L)=c_\rho x(L)=c_\rho(L)$ and thus $c_\rho$ preserves $L$; and similarly $\ol{c}_\rho$ preserves $L$ too. We now deal with cases.

    CASE 1: For every element $x\in\pi_1(\cG,v)$ as in (\ref{eq:elements}), $x$ acts on $T$ elliptically. Call $\cD=\cD_1\sqcup\dots\sqcup\cD_k$ the GBS core of $\cG$ and assume that $\rho$ belongs to $\cD_1$. Then this forces $\pi_1(\cD_1)$ to be isomorphic to $\bbZ$. In particular, we obtain that $c_\rho$ must belong to $C_\cG(\rho) = C_{\cD_1}(\rho) \cong \bbZ$ (by Proposition \ref{prop:comparison-of-centralizers}).

    CASE 2: There are elements $x\in\pi_1(\cG,v)$ as in (\ref{eq:elements}) and acting hyperbolically on $T$, but all of them have the same axis. Call $\cD=\cD_1\sqcup\dots\sqcup\cD_k$ the GBS core of $\cG$ and assume that $\rho$ belongs to $\cD_1$. Then this forces $\pi_1(\cD_1)$ to be isomorphic to $\bbZ^2$ or $K$ (the Klein bottle group). In particular, we obtain that $c_\rho$ must belong to $C_\cG(\rho) = C_{\cD_1}(\rho) \cong \bbZ$ or $\bbZ^2$ (by Proposition \ref{prop:comparison-of-centralizers}).

    CASE 3: There are at least two elements $x,x'\in\pi_1(\cG,v)$ as in (\ref{eq:elements}), and acting hyperbolically on $T$ with distinct axes $L,L'$. If $L\cap L'$ is a half-line, then $c_\rho,\ol{c}_\rho$ are inverse of each other and both preserve the half-line, and thus they must act as the identity on the half-line. If $L\cap L'$ is a segment, then $c_\rho$ preserves the segment and acts without inversions; therefore $c_\rho$ must act as the identity on the segment. If $L\cap L'$ is empty, then we take the unique shortest segment joining $L$ with $L'$, and $c_\rho$ must fix each point of that segment. In all cases, we find a point $t\in T$ such that $c_\rho$ lies in the stabilizer $H$ of $t$. It follows that $c_\rho$ must lie in $H\cap C(\rho)\cong\bbZ$.
    
    In all cases, we deduce that $c_\rho$ must lie in a subgroup of $\pi_1(\cG,v)$ which is isomorphic to either $\bbZ$ or $\bbZ^2$. It follows that $\ker\tau$ is a subgroup of a finitely generated free abelian group, and thus it must be a finitely generated free abelian group itself.
\end{proof}

\begin{remark}
    In the above discussion, we always twist all the edges glued at a certain root $\rho$ by the same element $c_\rho\in C(\rho)$. However, it might seem that the definition also works if we take a different twisting element $c_e\in C(\rho)$ for every edge $e$ glued onto $\rho$. In fact, a choice of twisting elements $c_e$ different for each edge actually defines a relative endomorphism $\Theta$, but the proof of Lemma \ref{lem:identity-on-centralizers} fails, and $\Theta$ might have no inverse. A counterexample is given by the Baumslag-Solitar group $\mathrm{BS}(2,4)=\pres{a,e}{a^2e=ea^4}$ by choosing distinct twisting elements $c_e,c_{\ol{e}}\in C(a)$ given by $c_{\ol{e}}=1$ and $c_e=ea\ol{e}\, \ol{a}$, obtaining the relative endomorphism $\Theta$ with $\Theta(a)=a$ and $\Theta(e)=eae\ol{a} \, \ol{e}$, which fails to be surjective.
\end{remark}

\begin{remark}\label{rem:centralizer_auto_are_inner_in_GBS}
Let $\cD=\cD_1\sqcup\dots\sqcup\cD_k$ be the GBS core of $\cG$. From Remark \ref{rem:embedding_GBS_in_group}, we have that, given $\rho\in V(\Delta_i)\subseteq\cR(\cG)$ with $\rho\in G_v$ for $v\in V(\Gamma)$, we have an embedding $G_i=\pi_1(\cD_i,\rho)\rar\pi_1(\cG,v)$. For each element $c_\rho\in C_\cG(\rho)$ in the centralizer of $\rho$, the relative endomorphism $\Theta:\FG{\cG}\rar\FG{\cG}$ defined as $\Theta(e)=c_{\rho_{\ol{e}}}e\ol{c}_{\rho_e}$ for $e\in E(\Gamma)$, restricts to an automorphism $\theta|_{G_i}: G_i \to G_i$ by Proposition \ref{prop:comparison-of-centralizers}, which is in fact an inner automorphism of $G_i$, i.e. $[\theta_{G_i}]$ is trivial in $\outr{G_i}$.
\end{remark}

\subsection{Reduction to the automorphism groups of GBSs}

\begin{thm}\label{thm:description-outrel}
    Let $\fF$ be a family of finitely presented torsion-free groups with cyclic centralizers. Let $\cG$ be a graph of groups with vertex groups in $\fF$ and with infinite cyclic edge groups. Suppose that $G=\pi_1(\cG)$ is freely indecomposable. Then there is an exact sequence
    \[
        1\rar Z\rar\prod_{\rho\in\cR}C(\rho)\xrightarrow{\tau} \outr{G}\xrightarrow{\sigma}\outr{G_1}\times\outr{G_2}\times\dots \times\outr{G_s}\times(D_\infty)^t \to 1
    \]
    for some finite set of elliptic elements $\cR\subseteq G$, for some finitely generated free abelian group $Z$, for some integers $s,t\ge0$, for some GBSs $G_1,\dots,G_s$ {\rm(}here $D_\infty$ is the infinite dihedral group{\rm)}.
\end{thm}
The elements appearing in the above Theorem \ref{thm:description-outrel} are defined as follows:
\begin{enumerate}
    \item We consider a JSJ decomposition $\cG'$ of $\cG$ relative to the family of its cyclic subgroups.
    \item We define $\cR=\notcyclicR{\cG'}$ to be a choice of non-cyclic roots for $\cG'$ (Definition \ref{def:choice-of-roots}). \item The homomorphism $\tau$ is the one defined in Section \ref{sec:auto-centralizers}.
    \item $Z=\ker\tau$ is the finitely generated free abelian group as in Proposition \ref{prop:kernel-free-abelian}.
    \item We define $\cD=\cD_1\sqcup\dots\sqcup\cD_k$ as the GBS core of $\cG'$ (as in Section \ref{sec:GBS-core}).
    \item We call $\cD_1,\dots,\cD_s$ the connected components with fundamental groups $G_i=\pi_1(\cD_i)$ not isomorphic to $\bbZ,\bbZ^2$ or $K$ (the Klein bottle group). This defines the integer $s\ge0$ and the groups $G_1,\dots,G_s$ appearing in the exact sequence.
    \item We call $\cD_{s+1},\dots,\cD_{s+t}$ the connected components with fundamental groups $\pi_1(\cD_i)$ isomorphic to $\bbZ^2$ or $K$ (the Klein bottle group). This defines the integer $t\ge0$ appearing in the exact sequence.
    \item We ignore the connected components $\cD_{s+t+1},\dots,\cD_k$ with fundamental groups $\pi_1(\cD_i)$ isomorphic to $\bbZ$ - these do not contribute to $\outr{G}$.
\end{enumerate}
In particular, if a cyclic JSJ decomposition for $G$ is given, then $\cR,s,t,G_1,\dots,G_s$ can be computed algorithmically.

\begin{remark}
    We recall again that in the statement of Theorem \ref{thm:description-outrel}, $C(\rho)$ denotes the centralizer of $\rho$ in $\pi_1(\cG)$, according to Definition \ref{def:centralizer-element}. This is isomorphic to the centralizer of $\rho$ in one of the GBSs $G_1,\dots,G_k$ (according to Proposition \ref{prop:comparison-of-centralizers}). This means that, in general, $C(\rho)$ will be an \textit{infinite GBS}, as described in Section \ref{sec:GBS-centralizers} (see Proposition \ref{prop:GBS-centralizer} and Figure \ref{fig:development}). In general $C(\rho)$ is not finitely generated.
\end{remark}

\begin{remark}
    In the above Theorem \ref{thm:description-outrel}, the condition of having cyclic centralizers inside the vertex groups can be weakened. In fact, we only need such a condition for the roots that appear in the cyclic JSJ decomposition $\cG'$. Note that these might be a different (and larger) set than the ones that appear in the initial given splitting $\cG$.
\end{remark}

\begin{proof}
    Let $\cG=(\Gamma,\{G_v\},\{G_e\},\{\psi_e\})$ be a graph of groups, and call $G=\pi_1(\cG)$ acting on the corresponding tree $T$. If $G$ is isomorphic to the fundamental group of a closed surface, then $\outr{G}$ is trivial and we are done; so we assume that this is not the case. Consider a JSJ decomposition for $\pi_1(\cG)$ over the family of its cyclic subgroups: vertex groups of such decomposition are either contained in a vertex group of $\cG$, or they are flexible, and thus isomorphic to a free group by \cite{GL17} (see \cite[Proposition 2.18]{ACK-iso1} for details); in any case, vertex groups have cyclic centralizers. Therefore, without loss of generality, we can assume that $\cG$ is a JSJ decomposition of $G$ over the family of its cyclic subgroups.

    Consider a choice of roots $\cR(\cG)$ (see Definition \ref{def:choice-of-roots}). For every $e\in E(\Gamma)$ we choose a generator $z_e$ for $G_e$. Up to changing the inclusion map $\psi_e:G_e\rar G_{\tau(e)}$ by conjugation, we can assume that $\psi_e(z_e)=\rho_e^{m_e}$ for some (unique) $\rho_e\in\cR(\Gamma)$ and $m_e\in\bbZ\setminus\{0\}$. Let also $\cD=(\Delta,\{H_v\}_{v\in V(\Delta)},\{H_e\}_{e\in E(\Delta)},\{\phi_e\}_{e\in E(\Delta)})$ be the GBS core of $\cG$, and call $\cD=\cD_1\sqcup\dots\sqcup\cD_k$ be its connected components. For $i=1,\dots,k$ we call $G_i:=\pi_1(\cD_i)$ acting on the corresponding tree $T_i$.

    Given an outer automorphism $\theta:\pi_1(\cG)\rar\pi_1(\cG)$ relative to $T$ (Definition \ref{def:relauto-tree}), we consider the corresponding relative endomorphism $\Theta:\FG{\cG}\rar\FG{\cG}$, given by Proposition \ref{prop:relendo-FG-pi1}, and defined up to equivalence (see Definition \ref{def:relendo-equivalence}). For $e\in E(\Gamma)$ we can take a reduced $(\iota(e),\tau(e))$-sound writing $\Theta(e)=g_0e_1g_1\dots g_{\ell-1}e_\ell g_\ell$, and we must have that $\rho_{\ol{e}}^{m_{\ol{e}}}\Theta(e)=\Theta(e)\rho_e^{m_e}$. But for $i=0,\dots,\ell$, by Lemma \ref{lem:sound-writing} and using the fact that we are working with groups with cyclic centralizers, we obtain that $g_i$ must be a power of $\rho_{e_i}=\rho_{\ol{e}_{i+1}}$. In particular, we can define a relative endomorphism $\Xi$ of $\cD$ (or, to be formal, relative endomorphisms $\Xi_i$ of $\cD_i$ for $i=1,\dots,k$) given by $\Xi(e)=g_0e_1g_1\dots g_{\ell-1}e_\ell g_\ell$ (the same reduced writing as $\Theta(e)$) for $e\in E(\Delta)$.

    Notice that different reduced sound writings for $\Theta(e)$ produce the same element $\Xi(e)$. Changing $\Theta$ by equivalence (according to Definition \ref{def:relendo-equivalence}) gives an equivalent relative endomorphism $\Xi$. In particular, we obtain a family of relative endomorphisms $\xi_i:\pi_1(\cD_i)\rar\pi_1(\cD_i)$ for $i=1,\dots,k$, (defined up to composing with conjugations). Notice that the composition of automorphisms $\theta$ translates into the composition of relative endomorphisms $\Theta$, which translates into the composition of relative endomorphisms $\Xi$, which translates into the composition of relative endomorphisms $\xi_i$ for $i=1,\dots,k$. In particular, since $\theta$ is invertible, so is each $\xi_i$. Therefore, we obtain a homomorphism $\sigma$
    \[
        \outrel{G}{T}\xrightarrow{\sigma} \outrel{G_1}{T_1}\times \outrel{G_2}{T_2}\times\dots \times\outrel{G_k}{T_k}
    \]
    and we want to compute $\ker\sigma$ and prove that $\sigma$ is surjective.

    SURJECTIVITY: Let $\xi_i:\pi_1(\cD_i)\rar\pi_1(\cD_i)$ be a family of relative endomorphisms, such that each of them admits an inverse. By Proposition \ref{prop:relendo-FG-pi1} we obtain a relative endomorphism $\Xi$ of $\cD$ (or, to be formal, relative endomorphisms $\Xi_i$ of $\cD_i$ for $i=1,\dots,k$), which admits an inverse up to equivalence. By Lemma \ref{lem:true-inverse} we have that $\Xi$ also has an actual inverse (not up to equivalence). We take a reduced $(\iota_\Delta(e),\tau_\Delta(e))$-sound writing $\Xi(e)=g_0e_1g_1\dots g_{\ell-1}e_\ell g_\ell\in\FG{\cD}$ and we define $\Theta(e)=g_0e_1g_1\dots g_{\ell-1}e_\ell g_\ell\in\FG{\cG}$. Note that different sound writings for $\Xi(e)$ produce the same element $\Theta(e)\in\FG{\cG}$. Note also that this defines a relative endomorphism $\Theta$ of $\cG$. It is easy to check that the inverse of $\Xi$ also defines a relative endomorphism of $\cG$, which is the inverse of $\Theta$ (here it is important that $\Xi$ has an actual inverse, and not only an inverse up to equivalence). By Proposition \ref{prop:relendo-FG-pi1}, we obtain a relative endomorphism $\theta:\pi_1(\cG)\rar\pi_1(\cG)$ which is invertible, and thus an element $\theta\in\outrel{G}{T}$. By definition we have $\sigma(\theta)=(\xi_i)_{i=1,\dots k}$ and this proves that $\sigma$ is surjective.

    KERNEL: Suppose that $\theta\in\ker\sigma$ for some $\theta\in\outrel{G}{T}$. Let $\Theta$ be the relative endomorphism of $\cG$ corresponding to $\theta$. Let $\Xi$ be the relative endomorphism of $\cD$ defined as above. Since $\theta\in\ker\sigma$, we must have that $\Xi$ is equivalent to the identity. According to Definition \ref{def:relendo-equivalence}, this means that there are elements $c_\rho\in C_\cD(\rho)$ for $\rho\in V(\Delta)=\cR(\cG)$ such that $\Xi(e)=c_{\iota_\Delta(e)}e\ol{c}_{\tau_\Delta(e)}$. But by Proposition \ref{prop:comparison-of-centralizers} we have that $C_\cD(\rho)=C_\cG(\rho)$, and thus we obtain that $\Theta(e)=c_{\iota_\Delta(e)}e\ol{c}_{\tau_\Delta(e)}$. This means that $\Theta$ is an automorphism induced by centralizers, as defined in Section \ref{sec:auto-centralizers}, i.e. $\theta$ belongs to the image of $\tau$. On the other hand, it is immediate (using the definition of equivalence of relative endomorphisms) to check that the composition $\sigma\circ\tau$ is trivial, and therefore the kernel of $\sigma$ coincides with the image of $\tau$.

    The kernel of $\tau$ had already been computed in Proposition \ref{prop:kernel-free-abelian}, and proved to be a finitely generated free abelian group $Z$. Finally, we observe that $\outrel{G}{T}=\outr{G}$ since $\cG$ is a cyclic JSJ decomposition for $G$, Similarly, when $G_i\not\cong\bbZ,\bbZ^2,K$ (the Klein bottle group) we have that $\cD_i$ is a cyclic JSJ decomposition for $G_i$, and thus $\outrel{G_i}{T_i}=\outr{G_i}$. When $G_i\cong\bbZ^2,K$ one can explicitly compute $\outrel{G_i}{T_i}\cong D_\infty$ (the infinite dihedral group). When $G_i\cong\bbZ$ one can explicitly compute $\outrel{G_i}{T_i}=1$. The statement follows.
    \end{proof}

    \begin{remark}
        In the proof of Theorem \ref{thm:description-outrel}, when proving surjectivity, it might seem that we are defining a section of $\sigma$; however, this is not the case. In fact, the relative endomorphism $\Xi$ of $\cD$ is defined up to equivalence, and we are making a choice of a representative in the equivalence class. Different choices of representatives can lead to different relative endomorphisms $\Theta$ of $\cG$, and thus $\Theta$ is not uniquely defined (not even up to equivalence).
    \end{remark}
  
\begin{cor}[Graphs of torsion-free hyperbolic groups]\label{cor:graph_hyperbolic_group}
Let $\cG$ be a graph of groups with torsion-free hyperbolic vertex groups and with infinite cyclic edge groups. Suppose that $G=\pi_1(\cG)$ is freely indecomposable. Then there are exact sequences
     \[
    \begin{array}{c}
    1 \to \outr{G} \to \outz{G} \to \prod MCG(\Sigma_v)\times (\hbox{finite group})\times (\hbox{fg abelian});\\[0.2cm]
        1\rar Z \rar\prod_{\rho\in\cR}C(\rho) \xrightarrow{\tau} \outr{G}\xrightarrow{\sigma}\outr{G_1}\times\outr{G_2}\times\dots \times\outr{G_s}\times(D_\infty)^t \to 1;
    \end{array}
    \]
    for some finite set of elliptic elements $\cR\subseteq G$, for some finitely generated free abelian group $Z$, for some integers $s,t\ge0$, and for some GBSs $G_1,\dots,G_s$ {\rm(}here $D_\infty$ is the infinite dihedral group{\rm)}.

    In the first exact sequence, the product runs over the flexible vertex groups of the JSJ decomposition of $G$ {\rm(}which are fundamental groups of surfaces with boundary{\rm)}, the finite groups occur as outer automorphisms of the rigid and cyclic vertex groups, and the finitely generated abelian groups occur as outer automorphisms sending cyclic vertex groups to proper subgroups containing a conjugate of themselves.
\end{cor}

From the combination theorem of Bestvina and Feighn, see \cite{BF92}, we have that $G$ is hyperbolic if and only if all the associated GBSs $G_i$ are cyclic. In this case, the centralizers $C(\rho)$ are also cyclic, and the above exact sequences become
    \[
    \begin{array}{c}
    1 \to \outr{G} \to \outz{G} \to \prod MCG(\Sigma_v)\times (\hbox{finite group});\\[0.2cm]
        1\rar Z \rar\prod_{\rho\in\cR} \mathbb Z^\rho \xrightarrow{\tau} \outr{G} \to 1;
    \end{array}
    \]  
and so $\outr{G}$ is finitely generated abelian. Therefore, we recover the result of Sela for one-ended torsion-free hyperbolic groups, see Theorem \ref{thm:sela}.

\bibliographystyle{alpha}

\end{document}